\documentclass[12pt,reqno,draft]{article} 
\usepackage{amsmath,amssymb,amsthm,amsfonts, indentfirst, empheq}
\usepackage{enumerate,color,bm,graphicx,here}
\usepackage{multicol}

\makeatletter
\def\@cite#1#2{[{{\bfseries #1}\if@tempswa , #2\fi}]}
\renewcommand{\section}{%
\@startsection{section}{1}{\z@}
{0.5truecm plus -1ex minus -.2ex}%
{1.0ex plus .2ex}{\bfseries\large}}
\def\@seccntformat#1{\csname the#1\endcsname.\ }
\makeatother

\numberwithin{equation}{section} 
\theoremstyle{theorem}
\newtheorem{thm}{Theorem}[section]

\newtheorem{lem}[thm]{Lemma}

\theoremstyle{definition}
\newtheorem{df}{Definition}[section]
\newtheorem{remark}{Remark}[section]
\newtheorem*{prth1.1}{Proof of Theorem 1.1}
\newtheorem*{prth1.2}{Proof of Theorem 1.2}

\newcommand{\ep}{\varepsilon}
\newcommand{\pa}{\partial}
\newcommand{\R}{\mathbb{R}}
\newcommand{\N}{\mathbb{N}}
\newcommand{\cl}[1]{{\overline#1}}
\newcommand{\Ombar}{\cl{\Omega}}

\newcommand{\Tmaxe}{T_{{\rm max}, \ep}}

\newcommand{\ue}{u_{\ep}}
\newcommand{\uie}{u_{0\ep}}
\newcommand{\ve}{v_{\ep}}
\newcommand{\vie}{v_{0\ep}}

\newcommand{\uej}{u_{\ep_j}}
\newcommand{\vej}{v_{\ep_j}}

\newcommand{\io}{\int_{\Omega}}

\newcommand{\upto}{\nearrow}
\newcommand{\dwto}{\searrow}

\newcommand{\embd}{\hookrightarrow}
\newcommand{\ddt}{\frac{\rm d}{{\rm d}t}}
\newcommand{\intd}[1]{{\rm d}#1}

\newcommand{\wc}{\rightharpoonup}
\newcommand{\wsc}{\stackrel{\star}{\rightharpoonup}}

\newcommand{\three}{I\hspace{-1.2pt}I\hspace{-1.2pt}I}
\newcommand{\four}{I\hspace{-1.2pt}V}
\newcommand{\baru}{\overline{u_0}}

\newcommand{\psf}{\textsf{p}}
\newcommand{\qsf}{\textsf{q}}
\newcommand{\rsf}{\textsf{r}}
\newcommand{\ssf}{\textsf{s}}

\DeclareMathOperator{\sgn}{sgn}

\DeclareBoldMathCommand{\bmass}{1}
\begin{document}
\footnote[0]
    {2020{\it Mathematics Subject Classification}\/. 
    Primary: 35B45; Secondary: 35Q92, 35D30, 92C17.
    }
\footnote[0]
    {{\it Key words and phrases}\/: 
    chemotaxis; flux limitation; weak solutions; nonlinear production.}
\begin{center}
    \Large{{\bf Global boundedness of weak solutions to\\ a flux-limited Keller--Segel system\\ with superlinear production}}
\end{center}
\vspace{5pt}
\begin{center}
    Shohei Kohatsu
   \footnote[0]{
    E-mail: 
    {\tt sho.kht.jp@gmail.com}
    }\\
    \vspace{12pt}
    Department of Mathematics, 
    Tokyo University of Science\\
    1-3, Kagurazaka, Shinjuku-ku, 
    Tokyo 162-8601, Japan\\
    \vspace{2pt}
\end{center}
\begin{center}    
    \small \today
\end{center}

\vspace{2pt}
\newenvironment{summary}
{\vspace{.5\baselineskip}\begin{list}{}{%
     \setlength{\baselineskip}{0.85\baselineskip}
     \setlength{\topsep}{0pt}
     \setlength{\leftmargin}{12mm}
     \setlength{\rightmargin}{12mm}
     \setlength{\listparindent}{0mm}
     \setlength{\itemindent}{\listparindent}
     \setlength{\parsep}{0pt}
     \item\relax}}{\end{list}\vspace{.5\baselineskip}}
\begin{summary}
{\footnotesize {\bf Abstract.}
The flux-limited Keller--Segel system
\begin{align*}
    \begin{cases}
    u_t = \Delta u - \chi \nabla \cdot (u|\nabla v|^{p-2}\nabla v),
    \\
    v_t = \Delta v - v + u^{\theta}
    \end{cases}
\end{align*}
is considered
under homogeneous Neumann boundary conditions
in a bounded domain
$\Omega \subset \mathbb{R}^n$ $(n \in \mathbb{N})$.
In the case that $\theta \le 1$,
existence of global bounded weak solutions was
established in \cite{KoP}.
The purpose of this paper
is to prove that global bounded weak solutions
can also be constructed in the case $\theta > 1$ with
a smallness condition on $p$.
} 
\end{summary}
\vspace{10pt}

\newpage
\section{Introduction}\label{Sec:Intro}

This study is concerned with the
flux-limited Keller--Segel system with superlinear
production,
\begin{align}\label{Sys:Main}
  \begin{cases}
    u_t=\Delta u - \chi \nabla \cdot (u |\nabla v|^{p-2} \nabla v)
    & \quad\mbox{in}\ \Omega \times (0, \infty),
  \\
    v_t = \Delta v - v + u^{\theta}
    & \quad\mbox{in}\ \Omega \times (0, \infty),
  \\
    \nabla u \cdot \nu=\nabla v \cdot \nu=0
    & \quad\mbox{on}\ \pa\Omega \times (0, \infty),
  \\
    u(\cdot ,0) = u_0, \ v(\cdot, 0) = v_0
    & \quad\mbox{in}\ \Omega,
  \end{cases}
\end{align}
where $\Omega$ is a $C^{2+\delta}$ bounded domain
in $\R^n$ $(n\in\N)$ for some $\delta \in (0,1)$,
where $\chi > 0$, $p > 1$ and $\theta > 1$ are constants,
where $\nu$ is the outward normal vector to $\pa\Omega$,
and where $u_0$ and $v_0$ are given nonnegative functions.
Chemotaxis systems with gradient-dependent flux limitation have been
proposed in some biomathematical literature
(see e.g.\ \cite{BBNS-2010, PY-2018}) as a generalization of
the minimal Keller--Segel system (\cite{KS-1970}),
which describes the chemotactic movement of cells (with density $u$) toward
a higher concentration ($v$) of the chemical substance produced
by themselves.

As in the classical Keller--Segel systems,
the question whether solutions of \eqref{Sys:Main} blow up
has received great interest.
When $\theta = 1$, the value $p = \frac{n}{n-1}$ has been found
to play a critical exponent in the sense that solutions can exist globally
and remain bounded if $p < \frac{n}{n-1}$
(\cite{YL-2020}), whereas there exist radially symmetric solutions
of the parabolic--elliptic variant
which blow up at finite time when $p > \frac{n}{n-1}$
(\cite{T-2022, Ko1}).
For related studies on global existence
and blow-up, we refer to \cite{W-2022-PP, Z-2023, NT-2018, MVY-2023}
and the references therein.

Compared to the flux-limited Keller--Segel system,
global existence and finite-time blow-up of solutions to the
chemotaxis system
with nonlinear production,
\begin{align}\label{Sys:np}
    \begin{cases}
        u_t = \Delta u - \nabla \cdot(u \nabla v)
        + f(u),
        & \quad x \in \Omega, \ t > 0,
    \\
        \tau v_t = \Delta v - v + g(u),
        & \quad x \in \Omega, \ t > 0
    \end{cases}
\end{align}
with $g(u) \simeq u^{\theta}$,
have been much more extensively studied.
For instance, when
$f \equiv 0$,
solutions exist globally and are bounded, provided that
$\theta < \frac{2}{n}$
(\cite{LT-2016}), whereas if $\Omega$ is a ball
and $\theta > \frac{2}{n}$, then some radial solutions to
a simplified parabolic--elliptic system blow up in finite time
(\cite{W-2018}).
A similar blow-up result in the quasilinear case is also available in
\cite{L-2019-JMAA}.
We note that the boundedness result in \cite{LT-2016} involves both
the sublinear case ($\theta < 1$) and the superlinear case ($\theta > 1$)
if $n=1$,
which the proofs rely on different arguments: If $\theta < 1$, then
we know from the mass conservation $\io u = \io u_0$ that
$\|g(u)\|_{L^{\frac{1}{\theta}}(\Omega)}$ is uniformly bounded,
whence we also obtain a uniform bound for $\|\nabla v\|_{L^r(\Omega)}$
with some $r > 1$, which benefits boundedness of solutions.
On the other hand, in the case of superlinear production,
$L^1$ estimates for $u$ are not sufficient to achieve the bounds for
$\|g(u)\|_{L^{1}(\Omega)}$.
In the proof, they derived the inequality
\begin{align}\label{Intro:quantity}
    \ddt\left[\io u^{\theta}+\frac{1}{2}\io |\nabla v|^2\right]
    + \io u^{\theta}+\frac{1}{2}\io |\nabla v|^2 \le C,
\end{align}
to ensure that $\|g(u)\|_{L^{1}(\Omega)}$ is indeed uniformly bounded.
After all, when $\theta > 1$,
the coupled quantities likely in \eqref{Intro:quantity} instead of
$L^1$ boundedness of solutions
often play a crucial role
in the analysis of \eqref{Sys:np}.

Similar boundedness results in modified systems
are available up to now,
and they also include the case of superlinear production.
For example, when
$f(u) \simeq u-\mu u^{\kappa}$,
it was shown in \cite{NO-2018} that solutions of \eqref{Sys:np}
globally exist and bounded under the condition that
$\theta \le \max\{\frac{\kappa}{2},2\}$
and $\theta < \frac{n+2}{2n}(\kappa-1)$,
and later on, this has been extended to
$\theta < \kappa -1$, or $\theta = \kappa-1$ with sufficiently large
$\mu$ in \cite{ZWZ-2019}.
In these results,
they rely on the effect of logistic-type dampening,
which may provide some estimates for $u$ in $L^{\kappa}(\Omega)$ spaces
with $\kappa > 1$,
to overcome
superlinearity of production, and also maximal regularity properties
would play an important role.
We refer to
\cite{V-2021, FV-2021, L-2020, DL-2025}
for related works in various directions.

Returning to the flux-limited Keller--Segel system \eqref{Sys:Main},
we are interested in whether solutions
globally exist and remain bounded when $\theta \neq 1$.
In the previous study (\cite{KoP}), we established global boundedness of
weak solutions to \eqref{Sys:Main} provided that
\begin{align}\label{Intro_KoP}
    \begin{cases}
        p \in (1, \infty)
        & \mbox{if}\ 0 < \theta \le \dfrac{1}{n},
    \\
        p \in \left(1, \dfrac{n\theta}{n\theta-1}\right)
        & \mbox{if}\ \dfrac{1}{n} < \theta \le 1
    \end{cases}
\end{align}
holds.
As a related but important work,
the 2D Keller--Segel--Navier--Stokes system,
\[
    \begin{cases}
        \rho_t + u\cdot \nabla \rho = \Delta \rho
        - \nabla\cdot (\rho f(|\nabla c|^2) \nabla c)
        &\quad\mbox{in}\ \Omega\times (0,\infty),
        \\
        c_t + u\cdot\nabla c = \Delta c - c + g(\rho)
        &\quad\mbox{in}\ \Omega\times (0,\infty),
        \\
        u_t + (u\cdot\nabla)u + \nabla P = \Delta u + \rho \nabla \Phi,
        \quad u = 0
        &\quad\mbox{in}\ \Omega\times (0,\infty)
    \end{cases}
\]
with $f(|\nabla c|^2) \lesssim (1+ |\nabla c|^2)^{-\frac{\alpha}{2}}$ and
$g(\rho) \lesssim \rho(1+\rho)^{\beta - 1}$, was considered in
\cite{W-2023},
where they obtained global boundedness of solutions under the condition
that
\[
    \begin{cases}
        \alpha > (2\beta - 1)_{+} - 1
        &\mbox{if}\ 0 < \beta \le 1,
        \\
        \alpha > 1 - \dfrac{1}{2\beta - 1}
        &\mbox{if}\ \beta > 1.
    \end{cases}
\]
This seems to be the first result on global existence in flux-limited
chemotaxis systems with superlinear production.
A similar result on global existence in a related system was obtained
in \cite{JRT-2023, LP-2024}.

However, to the best of our knowledge,
global solvability and boundedness in the system
\eqref{Sys:Main} with superlinear production ($\theta > 1$)
in domains of arbitrary dimension are still left as an open problem.
The purpose of this paper is to establish global existence and
boundedness of solutions to \eqref{Sys:Main} with $\theta > 1$ and
without any restriction on dimensional $n \in \N$.

Before we state the result,
we shall first introduce the solution concept
which will be pursued in the sequel.

%
%
\begin{df}\label{df1}
Let $u_0, v_0 \in L^1(\Omega)$.
A pair $(u,v)$ of functions
\[
        u, v \in L^1_{\rm loc}([0, \infty); W^{1,1}(\Omega))
\]
is called a \emph{global weak solution of \eqref{Sys:Main}} if
\begin{enumerate}[{\rm (i)}]
\item\label{df1-1}
$u \ge 0$ and
$v \ge 0 \ \mbox{a.e.}\ \mbox{on}\ \Omega \times (0,\infty)$,
\item\label{df1-2}
$u^{\theta} \in L^1_{\rm loc}([0, \infty); L^1(\Omega))$ and
$u|\nabla v|^{p-2}\nabla v \in L^1_{\rm loc}(
\Ombar \times [0, \infty); \R^n)$,
\item\label{df1-3}
for any
$\varphi \in C_{\rm c}^{\infty}(\Ombar \times [0, \infty))$,
\begin{align*}
    &\int_0^{\infty}\io u \varphi_t + \io u_0 \varphi(\cdot, 0)
    = \int_0^{\infty}\io \nabla u \cdot \nabla \varphi
    - \chi \int_0^{\infty}\io
      u|\nabla v|^{p-2}\nabla v \cdot \nabla\varphi
\\ \intertext{and}
    &\int_0^{\infty} \io v\varphi_t + \io v_0 \varphi(\cdot, 0)
    = \int_0^{\infty}\io \nabla v \cdot \nabla\varphi
    + \int_0^{\infty}\io v\varphi
    - \int_0^{\infty}\io u^{\theta}\varphi
\end{align*}
are valid.
\end{enumerate}
\end{df}
%

Our main result reads as follows.

%
%
\begin{thm}\label{Thm:Main}
Let $n \in \N$ and $\theta>1$,
and assume that $p > 1$ satisfies
\begin{align}\label{Main_Assumption}
    p\in \left(1, \dfrac{n\theta}{n\theta-1}\right).
\end{align}
Then, letting
\begin{align}\label{Main:s}
    \begin{cases}
        s= \infty
        & \mbox{if}\ n=1\ \mbox{and}\ 
        p\ge \min\left\{2, \dfrac{2\theta+1}{2\theta-1}\right\},
    \\
        s \neq \infty,\ s > \max\{n, (n+2)(p-1)\}\ \mbox{and}\ 
        s \ge 2
        & \mbox{otherwise},
    \end{cases}
\end{align}
given any initial data $u_0$ and $v_0$ such that
\begin{align}\label{Main_Initial}
    \begin{cases}
        u_0 \in L^{\infty}(\Omega) \setminus \{0\},
        \quad u_0 \ge 0\ \mbox{a.e.}\ \mbox{on}\ \Omega,\\
        v_0 \in W^{1,s}(\Omega),\quad
        v_0 \ge 0\ \mbox{on}\ \Ombar,
    \end{cases}
\end{align}
one can find a global weak solution $(u,v)$ of \eqref{Sys:Main}
with the properties that
\begin{align}\label{Main_Regularity}
    \begin{cases}
        u \in L^{\infty}(0, \infty; L^{\infty}(\Omega))
        \cap L^2_{\rm loc}([0, \infty); W^{1,2}(\Omega))
        \cap C^0_{{\rm w}-\ast}([0, \infty); L^{\infty}(\Omega)),
        \\
        v \in L^{\infty}(0, \infty; W^{1,s}(\Omega))
        \cap C^0_{{\rm w}-\ast}([0, \infty); L^{\infty}(\Omega)),
    \end{cases}
\end{align}
and that
\begin{align}\label{Main_Boundedness}
    \sup_{t > 0}\,
      (\|u(t)\|_{L^{\infty}(\Omega)} + \|v(t)\|_{L^{\infty}(\Omega)}) < \infty.
\end{align}
\end{thm}

%
%
\begin{remark}
The conditions \eqref{Intro_KoP} and \eqref{Main_Assumption}
suggest that the number $p=\frac{n\theta}{n\theta - 1}$
would play a role of critical blow-up exponent in the system
\eqref{Sys:Main}.
Recently, in a related system, some blow-up phenomena were considered
in \cite{MC-2024}; however, it has to be left as an open
whether there exists a solution which blows up
in the case that $p > \frac{n\theta}{n\theta - 1}$.
\end{remark}

%
%
\begin{remark}
Following \cite[Section 4]{JRT-2023}, one can construct a global solution
of \eqref{Sys:Main} with higher regularity than in \eqref{Main_Regularity},
namely, there exists $\alpha \in (0,1)$ such that
\[
    u \in C^{\alpha, \frac{\alpha}{2}}_{\rm loc}
    (\Ombar \times (0, \infty))
    \quad\mbox{and}\quad 
    v \in C^{\alpha, \frac{\alpha}{2}}
    (\Ombar \times [0, \infty))
    \cap C^{2+\alpha, 1+\frac{\alpha}{2}}_{\rm loc}
    (\Ombar \times (0, \infty)).
\]
In this situation, we can further obtain the boundedness property as
\[
    \sup_{t > 0}\,
    (\|u(\cdot, t)\|_{L^{\infty}(\Omega)} + \|v(\cdot, t)\|_{C^{\alpha}(\Ombar)})
    < \infty,
\]
noting that the assumption $s > n$ in \eqref{Main:s} ensures that
$W^{1,s}(\Omega) \embd C^{\alpha}(\Ombar)$.
\end{remark}

%
%
\begin{remark}
As in \cite{KoP} and \cite[Section 4]{KoY1},
we can further verify
stability or asymptotic stability of
constant equilibria $(\baru, \baru^{\theta})$,
provided that
$p$ is sufficiently large for the latter.
\end{remark}

\noindent
{\bf Main Ideas.}
As we already mentioned, the major difference between the sublinear case
and the superlinear case is whether uniform bounds for
$\|u^{\theta}\|_{L^1(\Omega)}$ would be available or not from the mass
conservation property $\io u = \io u_0$.
From the second equation of \eqref{Sys:Main} we see that
the lack of this boundedness information also implies that we cannot
rely on $L^r$ bounds for $v$ and $\nabla v$ with any $r \ge 1$
at the beginning.
To overcome this difficulty, we aim to ensure $L^2$ boundedness
of $\nabla v$ under the hypothesis in Theorem~\ref{Thm:Main}.
Here we note that we cannot derive the inequality
\eqref{Intro:quantity} as in \cite{LT-2016},
because in controlling the term $\ddt\io |\nabla v|^2$
we need to combine with the integral $\io u^{q-2}|\nabla u|^2$,
where we cannot choose $q = \theta$ in most cases
(see Lemma~\ref{Lem:vexL2Func}).
Accordingly, we perform our analysis of the functional
\begin{align}\label{Intro_Funcq2}
    \io u^q + \io |\nabla v|^2
\end{align}
with some $q > 1$,
and derive a differential inequality of the form
\[
    \ddt\left[\io u^q + \io |\nabla v|^2\right]
    + \dfrac{q(q-1)}{8}\io u^{q-2}|\nabla u|^2 + 2\io |\nabla v|^2
    + \frac{1}{2} \io |\Delta v|^2 \le C
\]
(Lemma~\ref{Lem:vexL2}),
which corresponds to \eqref{Intro:quantity}.
The keys to this derivation are twofold.
Firstly, we need to derive an $L^2$ estimate for $v$, which will be one of the
ingredient to control flux-limited chemotactic term in \eqref{Sys:Main}
during the proof of Lemma~\ref{Lem:vexL2}.
To this end, we again rely on a functional which has the form
\[
    \sgn(q-1) \io u^q + C \io v^2
\]
with suitable choice of $q > 0$ for each $n \in \N$
(Lemmas~\ref{Lem:ueLqFunc1} and \ref{Lem:veL2Func})
and a constant $C>0$, and this leads to
another differential inequality that ensures $L^2$ boundedness
for $v$ with restriction on $p$
(Lemma~\ref{Lem:veL2}).
Secondly, throughout the analysis on \eqref{Intro_Funcq2},
the suitable selection of parameter $q$ is also required in order
to control the norm
$\|\nabla v\|_{L^r(\Omega)}^{\kappa(q)}$
with some $\kappa(q) > 0$ and $r > 1$, which arises in
the computation on the basis of the term $\ddt \io u^q$
 (Lemmas~\ref{Lem:ueLqFunc2} and \ref{Lem:vexLr}).
These two steps are important to extend the boundedness result
to arbitrary dimensional settings, and as a by-product,
we also obtain $L^q$ estimates for $u$ with any $q > 1$
(Lemma~\ref{Lem:ueLq}).
We note that
similar functionals are also used in
\cite{WL-2024}.

The paper is organized as follows.
In Section~\ref{Sec:Pre} we introduce approximate problems, and
collect some inequalities.
Section~\ref{Sec:BddApp} is devoted to global boundedness of
approximate solutions.
Finally, in Section~\ref{Sec:Proof} we consider convergence of
approximate solutions and prove Theorem~\ref{Thm:Main}.

\section{Preliminaries}\label{Sec:Pre}

In this section we first introduce approximate problems
and state an existence result,
and then recall some inequalities which will be useful
in Section~\ref{Sec:BddApp}.
Throughout the sequel, we shall assume that
$u_0$ and $v_0$ fulfill \eqref{Main_Initial}.

\subsection{Local existence of approximate solutions}\label{Subsec:local}

We consider the approximate problems
\begin{align}\label{Sys:App}
  \begin{cases}
    (\ue)_t=\Delta \ue - \chi \nabla
    \cdot\big(\ue(|\nabla \ve|^{2}+\ep)^{\frac{p-2}{2}}\nabla\ve),
    & x \in \Omega, \ t > 0,
  \\
    (\ve)_t = \Delta \ve - \ve + \ue^{\theta},
    & x \in \Omega, \ t > 0,
  \\
    \nabla \ue \cdot \nu=\nabla \ve \cdot \nu=0,
    & x \in \pa\Omega, \ t > 0,
  \\
    \ue(x, 0) = \uie(x), \ \ve(x, 0) = \vie(x),
    & x \in \Omega,
  \end{cases}
\end{align}
for $\ep \in (0,1)$, where
$\uie,\vie$ are nonnegative functions with the following properties:
\begin{enumerate}[{\rm (I)}]
\item\label{Ass1}
$\uie \in W^{1,\infty}(\Omega) \setminus\{0\}$
for all $\ep \in (0,1)$,
$\uie\to u_0$ in $L^{\infty}(\Omega)$
as $\ep \dwto 0$,
and moreover,
for any $q \in [1, \infty]$ there exists $U_q > 0$ satisfying
\begin{align}
    \|\uie\|_{L^q(\Omega)} \le U_q \|u_0\|_{L^q(\Omega)}
    \label{bound:uie}
\end{align}
for all $\ep \in (0,1)$.
\item\label{Ass2}
If $s=\infty$ in \eqref{Main:s},
then $\vie = v_0$ for every $\ep \in (0,1)$.
\item\label{Ass3}
If $s < \infty$ in \eqref{Main:s}, then
$\vie\in C^{1+\delta}(\Ombar)$
as well as $\nabla\vie \cdot \nu = 0$ on $\pa\Omega$
for each $\ep\in (0,1)$,
$\vie \to v_0$ in $W^{1,s}(\Omega)$ as $\ep \dwto 0$,
and furthermore,
for each $r\in [1,s]$ there is $V_r > 0$ fulfilling
\begin{align}
    \|\vie\|_{W^{1,r}(\Omega)} \le V_r \|v_0\|_{W^{1,r}(\Omega)}
    \label{bound:vie}
\end{align}
for any $\ep \in (0,1)$.
\end{enumerate}
We note that such families of functions
$(\uie)_{\ep \in (0,1)}$ and $(\vie)_{\ep \in (0,1)}$
can be constructed by a straightforward mollifying procedure
(see e.g.\ \cite[Proposition 4.21 and Corollary 9.8]{B-2011}
and \cite[Section 2]{STY-2024}).

Let us begin with
a theory of local existence in \eqref{Sys:App}, which can be
proved by making use of
the Schauder fixed point theorem along with
standard parabolic theory.
This has been used in previous studies
on related problems (see e.g.\ \cite{DW-2019, ILV-2024}),
however, due to the lack of smoothness of the boundary
$\pa\Omega$
a small modification is required in our proof.

%
%
\begin{lem}\label{Lem:local_App}
For any $\ep \in (0,1)$, there exist
$\Tmaxe\in(0, \infty]$ and a
uniquely determined pair $(\ue,\ve)$ of functions
\begin{align}\label{local_App:reg}
    \begin{cases}
        \ue, \ve \in C^0(\Ombar \times [0, \Tmaxe))
          \cap C^{2,1}(\Ombar \times (0, \Tmaxe)),
          \quad\mbox{and moreover}, \\
        \ve \in \bigcap_{q>n} C^0([0, \Tmaxe); W^{1,q}(\Omega))
          \quad\mbox{if}\ s = \infty,\\
          \ve \in C^0([0, \Tmaxe); W^{1,s}(\Omega))
          \quad\hspace{9.5mm}\mbox{if}\ s < \infty,
    \end{cases}
\end{align}
such that $\ue > 0$ and $\ve > 0$ in
$\Ombar \times (0, \Tmaxe)$, that $(\ue, \ve)$ solves
\eqref{Sys:App} classically in $\Omega \times (0, \Tmaxe)$,
and that
\begin{align}\label{local_App:criterion}
    \mbox{if}\ \Tmaxe<\infty, \quad
    \mbox{then}\quad
    \limsup_{t\upto \Tmaxe} \left\{
      \|\ue(\cdot,t)\|_{L^{\infty}(\Omega)}
      + \|\ve(\cdot, t)\|_{
      W^{1,s}(\Omega)}\right\}
      = \infty.
\end{align}
Moreover,
\begin{align}\label{local_App:mass}
    \io \ue(\cdot,t) = \io \uie
    \quad\mbox{for all}\ t \in (0, \Tmaxe)
\end{align}
holds for every $\ep \in (0,1)$.
\end{lem}
\begin{proof}
For each $\ep \in (0,1)$,
existence of solutions $(\ue,\ve)$
in a weak sense
together with \eqref{local_App:criterion}
can be verified in the same way as in
\cite[Lemma 2.1]{W-2019},
however, during the proof, we apply
\cite[Theorem 1.1]{L-1987} instead
to achieve $L^s$ bounds for $\nabla \ve$ on
$\Ombar \times (0, T)$ with some small $T>0$,
and also refer to
\cite[Theorem~\three.1.3 and Remark~\three.1.1]{D-1993}
to ensure H\"{o}lder continuity of solutions.
Similarly as in \cite[Lemma 4.7]{Ko3},
the regularity \eqref{local_App:reg} follows from
parabolic regularity theory
(\cite[Theorem~\four.5.3]{LSU-1968}),
and thus $(\ue, \ve)$ is a classical solution of
\eqref{Sys:App} in $\Omega \times (0, \Tmaxe)$.
Moreover, uniqueness of classical solutions can be
obtained by a standard testing argument.
Furthermore, positivity of solutions follows from strong maximal
principle (cf.\ \cite[Proposition 52.7]{QS-2019}),
and finally the mass conservation property
\eqref{local_App:mass} immediately results by
integrating the first equation in \eqref{Sys:App}
over $\Omega$.
\end{proof}

%
%
\begin{remark}
If $n=1$,
then by virtue of \cite[Theorems 14.4, 14.6 and 15.5]{A-1993}
the solution $(\ue,\ve)$ from Lemma~\ref{Lem:local_App} satisfies
\[
    \ue, \ve \in \bigcap_{q>1}
    C^0([0, \Tmaxe); W^{1,q}(\Omega)),
\]
and fulfills \eqref{local_App:criterion} with replacing
$\|\ve(\cdot,t)\|_{W^{1,s}(\Omega)}$ by
$\|\ve(\cdot,t)\|_{L^{\infty}(\Omega)}$.
\end{remark}

\subsection{Various estimates}

During estimates for solutions,
we will rely on the following results from
the Gagliardo--Nirenberg interpolation inequality.

%
%
\begin{lem}\label{Lem:GN1}
Suppose that $\psf, \qsf, \rsf, \ssf$ satisfy
one of the following\/{\rm :}
\begin{enumerate}[{\rm (i)}]
\item\label{GN1-1}
$\psf \in (0, \infty]$, $\qsf\in (0, \psf]$,
$\rsf \in [1, \infty)$, $\ssf \in (0, \infty)$ and
\begin{align}\label{GN1-1:a}
    a := \frac{\frac{1}{\qsf} - \frac{1}{\psf}}
      {\frac{1}{\qsf} + \frac{1}{n} - \frac{1}{\rsf}} \in [0,1];
\end{align}
%
\item\label{GN1-2}
$\psf \in [1, \infty)$, $\qsf \in [1, \infty]$,
$\rsf \in [1, \infty]$ and $\ssf \in (0, \infty)$ such that
\eqref{GN1-1:a} holds.
\end{enumerate}
Then there exists a constant
$C = C(\Omega, n, \psf, \qsf, \rsf) > 0$ such that
\[
    \|f\|_{L^{\psf}(\Omega)}
    \le C \|\nabla f\|_{L^{\rsf}(\Omega)}^{a}
    \|f\|_{L^{\qsf}(\Omega)}^{1-a}
    + C \|f\|_{L^{\ssf}(\Omega)}
\]
for all $f \in W^{1, \rsf}(\Omega) \cap L^{\qsf}(\Omega)$.
\end{lem}
\begin{proof}
The case \eqref{GN1-1} was proved in
\cite[Lemma 2.3]{LL-2016}, while
the case \eqref{GN1-2} results from
\cite[p.126, comment 5]{N-1959}.
\end{proof}

In order to state the result of the interpolation inequality
in the case of higher-order regularity
we will introduce the space
\[
    W^{2,2}_{\rm N}(\Omega) :=
    \{f \in W^{2,2}(\Omega) \mid \nabla f \cdot \nu = 0
    \ \mbox{on}\ \pa\Omega\}.
\]
%

%
%
\begin{lem}\label{Lem:GN2}
Let $\psf \in [1, \infty)$, $\qsf \in [1, \infty]$,
$\rsf \in [2, \qsf]$ and $\ssf \in (0, \infty)$,
and assume that
\[
    b := \frac{\frac{1}{\qsf}+\frac{1}{n}-\frac{1}{\psf}}
    {\frac{1}{\qsf}+\frac{2}{n}-\frac{1}{2}}
    \in \left[\frac{1}{2},1\right].
\]
Then there exists a constant
$C = C(\Omega, n, \psf, \qsf, \rsf) > 0$ such that
\[
    \|\nabla f\|_{L^{\psf}(\Omega)}
    \le C (\|\Delta f\|_{L^{2}(\Omega)}^{b}
      + \|f\|_{L^{\rsf}(\Omega)}^b)
    \|f\|_{L^{\qsf}(\Omega)}^{1-b}
    + C \|f\|_{L^{\ssf}(\Omega)}
\]
for all $f \in W^{2,2}_{\rm N}(\Omega) \cap L^{\qsf}(\Omega)$.
\end{lem}
\begin{proof}
Thanks to \cite[p.126, comment 5]{N-1959}, there exists
$c_1 = c_1(\Omega, n, \psf, \qsf) > 0$ such that
\begin{align}\label{GN2:GN}
    \|\nabla f\|_{L^{\psf}(\Omega)}
    \le c_1 \|D^2 f\|_{L^2(\Omega)}^b
    \|f\|_{L^{\qsf}(\Omega)}^{1-b}
    + c_1 \|f\|_{L^{\ssf}(\Omega)}
\end{align}
for all $f \in W^{2,2}(\Omega) \cap L^{\qsf}(\Omega)$.
Moreover, standard elliptic regularity theory
(see \cite[Theorem 9.26]{B-2011} for instance) along with
the H\"{o}lder inequality provides a constant
$c_2 = c_2(\Omega, r) > 0$ such that
\[
    \|D^2 f\|_{L^2(\Omega)}
    \le c_2 (\|\Delta f\|_{L^2(\Omega)} + \|f\|_{L^{\rsf}(\Omega)})
\]
for all $f \in W^{2,2}_{\rm N}(\Omega) \cap L^{\qsf}(\Omega)$.
This together with \eqref{GN2:GN} thereby proves the claim.
\end{proof}

Let us close this section by collecting
some results on
$L^{\psf}$-$L^{\qsf}$ estimates for the Neumann heat
semigroup on bounded intervals.
Throughout the sequel,
we let $(e^{\tau \Delta})_{\tau \ge 0}$
denote the Neumann heat
semigroup in $\Omega$,
and
we let $\lambda_1 > 0$ represent
the first nonzero eigenvalue of $-\Delta$ on $\Omega$
under Neumann boundary conditions.

%
%
\begin{lem}\label{Lem:semigroup}
Let $n=1$.
Then there exist constants $C_1, C_2, C_3
> 0$
depending only on $\Omega$ with the following
properties\/{\rm :}
\begin{enumerate}[{\rm (i)}]
\item\label{semigroup-1}
If $\psf, \qsf$ satisfy
$1 \le \qsf \le \psf \le \infty$, then
\[
    \|\nabla e^{\tau \Delta} w\|_{L^{\psf}(\Omega)}
    \le C_1 (1+ \tau^{-\frac{1}{2}-\frac{1}{2}(
      \frac{1}{\qsf}-\frac{1}{\psf})})
    e^{-\lambda_1 \tau} \|w\|_{L^{\qsf}(\Omega)}
    \quad\mbox{for all}\ \tau > 0
\]
holds for any $w \in L^{\qsf}(\Omega)$.
\item\label{semigroup-2}
If $\psf, \qsf$ fulfill $2 \le \qsf \le \psf \le \infty$, then
\[
    \|\nabla e^{\tau \Delta}w\|_{L^{\psf}(\Omega)}
    \le C_2 
    (1+\tau^{-\frac{1}{2}(\frac{1}{\qsf}-\frac{1}{\psf})})
    e^{-\lambda_1 \tau} \|\nabla w\|_{W^{1, \qsf}(\Omega)}
    \quad\mbox{for all}\ \tau > 0
\]
is valid for all $w \in W^{1,\qsf}(\Omega)$.
\item\label{semigroup-3}
Let $\psf, \qsf$ be such that
$1 < \qsf \le \psf < \infty$ or
$1 < \qsf < \psf = \infty$.
Then
\[
    \|e^{\tau \Delta}\nabla w\|_{L^{\psf}(\Omega)}
    \le C_3 (1+ \tau^{-\frac{1}{2}-\frac{1}{2}(
      \frac{1}{\qsf}-\frac{1}{\psf})})
    e^{-\lambda_1 \tau} \|w\|_{L^{\qsf}(\Omega)}
    \quad\mbox{for all}\ \tau > 0
\]
holds for any $w \in L^{\qsf}(\Omega)$,
where $e^{\tau \Delta}\nabla$
is the extension of the corresponding operator
on $C_{\rm c}^{\infty}(\Omega)$ to a continuous operator
from $L^{\qsf}(\Omega)$ to $L^{\psf}(\Omega)$.
\end{enumerate}
\end{lem}
\begin{proof}
We refer to \cite[Lemma 1.3~(ii)--(iv)]{W-2010}
and \cite[Lemma 2.1~(ii)--(iv)]{C-2015} for the proofs.
We note that
the case \eqref{semigroup-2} with $\psf = \infty$ is
not contained in \cite{W-2010, C-2015},
however, we can proceed in the same way as in \cite{C-2015}
to obtain the desired estimate, since
the Poinc\'{a}re--Wirtinger inequality is applicable also
in this case.
\end{proof}

\section{Global boundedness in approximate problems}\label{Sec:BddApp}

The purpose of this section is to establish
$\ep$-independent bounds for approximate solutions,
which will be fundamental ingredients for the proof of
Theorem~\ref{Thm:Main}.
Throughout the remaining part of this paper,
for each $\ep \in (0,1)$ we let $(\ue, \ve)$ denote the
approximate solutions obtained in Lemma~\ref{Lem:local_App}
with $\Tmaxe \in (0, \infty]$ representing their maximal
existence time.

As usual, it will be important to achieve $L^q$ estimates for
$\ue$ with $q > 1$.
It is known that in the case of sublinear production
($\theta \le 1$),
the mass conservation feature \eqref{local_App:mass}
immediately provides some $L^{\frac{1}{\theta}}$ bounds
for the term
$\ue^{\theta}$ in the second equation of \eqref{Sys:App},
which together with semigroup estimates makes it
possible to gain
$L^r$ estimates for $\nabla \ve$ with
$1 \le r < \frac{n}{n\theta - 1}$
(cf.\ \cite[Lemma 2.2]{KoP} and \cite[Theorem 1.1]{MC-2024}).

Since we cannot expect such situations in
our case,
in the first step toward our analysis we will rely on some properties
of the functionals $\sgn(q-1)\io \ue^q$ for $q > 0$
under the restriction that $p < 2$,
which will be used to obtain $L^2$ estimates for $\ve$ in
Lemma~\ref{Lem:veL2}.

%
%
\begin{lem}\label{Lem:ueLqFunc1}
Let $n\in\N$, and
let $p>1$ satisfy
%
\[
    p \in \left(1, \min\left\{2, 1+\frac{2}{n}\right\}\right).
\]
%
Suppose that $q > 0$ fulfills
\begin{align}\label{ueLqFunc1:q}
    q \in \left(\max\left\{0, 1-\frac{2}{n}\right\}, 
        \frac{2(2-p)}{n(p-1)}\right)\setminus\{1\}.
\end{align}
Then for any $\eta > 0$ there exists a constant $C>0$ such that
\begin{align}
\nonumber
    &\sgn(q-1) \ddt \io \ue^q(\cdot,t)
    + \left(\frac{q|q-1|}{2} - \eta\right) \io
      \ue^{q-2}(\cdot,t)|\nabla \ue(\cdot,t)|^2
\\
    &\quad\,\le C\io |\nabla \ve(\cdot,t)|^2 + C
    \label{ueLqFunc1:ddtueLq}
\end{align}
for all $t \in (0, \Tmaxe)$ and $\ep \in (0,1)$.
\end{lem}
\begin{proof}
Testing the first equation in \eqref{Sys:App} against
$\sgn(q-1)\ue^{q-1}$,
we see that due to the Young inequality,
\begin{align}
\nonumber
    \frac{\sgn(q-1)}{q}\ddt \io \ue^q
    &= -|q-1| \io \ue^{q-2}|\nabla \ue|^2
    + \chi|q-1| \io \ue^{q-1} (|\nabla \ve|^2+\ep)^{\frac{p-2}{2}}
    \nabla \ue \cdot \nabla \ve
\\
    &\le -\frac{|q-1|}{2} \io \ue^{q-2}{|\nabla \ue|^2}
    + \frac{\chi^2 |q-1|}{2} \io \ue^q |\nabla \ve|^{2(p-1)}
    \label{ueLqFunc1:ineq1}
\end{align}
in $(0, \Tmaxe)$ for all $\ep \in (0,1)$.
Here since $p < 2$, we may use
the H\"{o}lder inequality to estimate
\begin{align}\label{ueLqFunc1:ineq2}
    \io \ue^q |\nabla \ve|^{2(p-1)}
    \le \left(\io \ue^{\frac{q}{2-p}}\right)^{2-p}
    \left(\io |\nabla \ve|^2\right)^{p-1}
\end{align}
in $(0,\Tmaxe)$ for each $\ep \in (0, 1)$.
Moreover, the assumptions
$p < 1+\frac{2}{n}$ and $q > 1-\frac{2}{n}$
enable us to apply the Gagliardo--Nirenberg inequality
(Lemma~\ref{Lem:GN1}~\eqref{GN1-1} or \eqref{GN1-2}) to find $c_1 > 0$ such that
\begin{align*}
    \left(\io \ue^{\frac{q}{2-p}}\right)^{2-p}
    &= \big\|\ue^{\frac{q}{2}}\big\|_{L^{\frac{2}{2-p}}(\Omega)}^2
\\
    &\le c_1 \big\|\nabla \ue^{\frac{q}{2}}\big\|_{L^2(\Omega)}^{2a_1}
      \big\| \ue^{\frac{q}{2}}\big\|_{L^{\frac{2}{q}}(\Omega)}^{2(1-a_1)}
      + c_1 \big\| \ue^{\frac{q}{2}}\big\|_{L^{\frac{2}{q}}(\Omega)}^2
\\
    &= \left(\frac{q^2}{4}\right)^{a_1} c_1
      \left(\io \ue^{q-2}|\nabla \ue|^2\right)^{a_1}
      \|\ue\|_{L^1(\Omega)}^{q(1-a_1)}
      + c_1 \|\ue\|_{L^1(\Omega)}^q
\end{align*}
in $(0, \Tmaxe)$ for any $\ep \in (0,1)$,
where $a_1 := \frac{\frac{q}{2} - \frac{2-p}{2}}{\frac{q}{2}
+ \frac{1}{n} - \frac{1}{2}} \in (0,1)$.
In light of \eqref{local_App:mass} and \eqref{bound:uie},
inserting this and \eqref{ueLqFunc1:ineq2}
into \eqref{ueLqFunc1:ineq1}
and applying the Young inequality,
we infer that for arbitrary fixed $\eta > 0$
there exists $c_2 > 0$ satisfying
\begin{align}
\nonumber
    &\sgn(q-1) \ddt \io \ue^q 
    + \frac{q|q-1|}{2} \io \ue^{q-2}|\nabla \ue|^2
\\
    &\quad\,
      \le\eta \io \ue^{q-2}|\nabla \ue|^2
      + c_2 \left(\io |\nabla \ve|^2\right)^{
        \frac{p-1}{1-a_1}}
    + c_2 \left(\io |\nabla \ve|^2\right)^{p-1}
    \label{ueLqFunc1:ineq3}
\end{align}
in $(0, \Tmaxe)$ for all $\ep \in (0,1)$.
Furthermore, since the relation $a_1 < 1$ and
the condition \eqref{ueLqFunc1:q} ensure that
\[
    p-1 < \frac{p-1}{1-a_1}
    = \frac{(\frac{q}{2} + \frac{1}{n} - \frac{1}{2})(p-1)}
      {\frac{1}{n} + \frac{1}{2} - \frac{p}{2}}
    < \frac{\frac{2+n-np}{2n(p-1)}(p-1)}
      {\frac{1}{n} + \frac{1}{2} - \frac{p}{2}}
    = 1,
\]
once more using the Young inequality we have
\begin{align}\label{ueLqFunc1:vex1}
    &\left(\io |\nabla \ve|^2\right)^{
        \frac{p-1}{1-a_1}}
    \le \io |\nabla \ve|^2 + 1
\\ \intertext{and}
\nonumber
    &\left(\io |\nabla \ve|^2 \right)^{p-1}
      \le \io |\nabla \ve|^2 + 1
\end{align}
in $(0, \Tmaxe)$ for each $\ep \in (0,1)$.
Together with \eqref{ueLqFunc1:ineq3} and \eqref{ueLqFunc1:vex1}
this establishes \eqref{ueLqFunc1:ddtueLq} with $C=2c_2$.
\end{proof}

In order to gain $L^2$ estimates for $\ve$
we next rely on a standard testing
procedure associated with the second equation in
\eqref{Sys:App}.

%
%
\begin{lem}\label{Lem:veL2Func}
Let $n\in \N$ and $\theta > 1$,
and assume that $q > 0$ satisfies
\begin{align}\label{veL2Func:q}
    q > 
    \max\left\{0, 2\theta-\frac{4}{n}, 2\theta-1-\frac{2}{n}\right\}.
\end{align}
Then for all $\eta > 0$ there exists a constant $C>0$ such that
\begin{align}\label{veL2Func:ddtveL2}
    \ddt \io \ve^2(\cdot,t) + \io \ve^2(\cdot,t) 
    + \io |\nabla\ve(\cdot,t)|^2
    \le \eta \io \ue^{q-2}(\cdot,t)|\nabla \ue(\cdot,t)|^2
    + C
\end{align}
for all $t \in (0, \Tmaxe)$ and $\ep \in (0,1)$.
\end{lem}
\begin{proof}
We first note that since \eqref{veL2Func:q} implies that
\begin{align*}
    &q > 2\theta + 1 - \frac{2}{n} 
    - 2\min\left\{1, \frac{n+2}{2n}\right\}
\\ \intertext{and}
    &q > \frac{2\theta(n-2)}{n+2}
    = \left(1 - \frac{2}{n}\right)\theta\frac{2n}{n+2},
\end{align*}
there exists $r > \max\{1, \frac{2n}{n+2}\}$ fulfilling
\begin{align}
    &q > 2\theta + 1 - \frac{2}{n} - \frac{2}{r}
    \label{veL2Func:Young}
\\ \intertext{as well as}
    &q > \left(1 - \frac{2}{n}\right)\theta r.
    \label{veL2Func:GNr}
\end{align}
Keeping this value of $r$ fixed henceforth,
we now multiply the second equation in \eqref{Sys:App} by
$\ve$
and make use of the H\"{o}lder inequality,
the embedding
$W^{1,2}(\Omega) \embd L^{\frac{r}{r-1}}(\Omega)$
and the Young inequality to find $c_1 > 0$ such that
\begin{align}
\nonumber
    \frac{1}{2}\ddt \io \ve^2 + \io |\nabla\ve|^2
    + \io \ve^2 &= \io \ue^{\theta} \ve
\\ \nonumber
    &\le \left(\io \ue^{\theta r}\right)^{\frac{1}{r}}
      \|\ve\|_{L^{\frac{r}{r-1}}(\Omega)}
\\ \nonumber
    &\le c_1 \left(\io \ue^{\theta r}\right)^{\frac{1}{r}}
      \left(\io \ve^2 + \io |\nabla \ve|^2\right)^{\frac{1}{2}}
\\
    &\le \frac{1}{2} \io \ve^2 + \frac{1}{2} \io |\nabla \ve|^2
      + \frac{c_1}{2} \left(\io \ue^{\theta r}\right)^{\frac{2}{r}}
    \label{veL2Func:ineq1}
\end{align}
in $(0,\Tmaxe)$ for all $\ep \in (0,1)$.
Moreover, the relation \eqref{veL2Func:GNr} warrants that
the Gagliardo--Nirenberg inequality (Lemma~\ref{Lem:GN1}~\eqref{GN1-1})
becomes applicable so as to yield $c_2 > 0$ such that
\begin{align}
\nonumber
    \left(\io \ue^{\theta r}\right)^{\frac{2}{r}}
    &= \big\|\ue^{\frac{q}{2}}\big\|_{
      L^{\frac{2\theta r}{q}}(\Omega)}^{
        \frac{4\theta}{q}}
\\ \nonumber
    &\le c_2 \big\|\nabla \ue^{\frac{q}{2}}\big\|_{L^2(\Omega)}^{
      \frac{4a_1 \theta}{q}}
    \big\|\ue^{\frac{q}{2}}\big\|_{L^{\frac{2}{q}}(\Omega)}^{
      \frac{4(1-a_1)\theta}{q}}
    + c_2 \big\|\ue^{\frac{q}{2}}\big\|_{L^{\frac{2}{q}}(\Omega)}^{
      \frac{4\theta}{q}}
\\
    &= \left(\frac{q^4}{16}\right)^{\frac{a_1 \theta}{q}}
      c_2 \left(\io \ue^{q-2}|\nabla \ue|^2\right)^{
        \frac{2a_1 \theta}{q}}
      \|\ue\|_{L^1(\Omega)}^{2(1-a_1)\theta}
    + c_2 \|\ue\|_{L^1(\Omega)}^{2\theta}
    \label{veL2Func:GN}
\end{align}
in $(0, \Tmaxe)$ for every $\ep \in (0,1)$, where
$a_1 := \frac{\frac{q}{2} - \frac{q}{2\theta r}}
{\frac{q}{2} + \frac{1}{n} - \frac{1}{2}} \in (0,1)$.
Here, \eqref{veL2Func:Young} ensures that
\[
    \frac{2a_1 \theta}{q}
    = \frac{\theta - \frac{1}{r}}{\frac{q}{2}+\frac{1}{n}-\frac{1}{2}}
    < \frac{\theta-\frac{1}{r}}
    {\frac{1}{2}(2\theta+1-\frac{2}{n}-\frac{2}{r})+\frac{1}{n}
    -\frac{1}{2}} = 1,
\]
so that we can utilize the Young inequality in the right of
\eqref{veL2Func:GN} along with
\eqref{local_App:mass} and \eqref{bound:uie} to observe that
for any fixed $\eta > 0$,
there exists $c_3 > 0$ such that
\[
    \left(\io \ue^{\theta r}\right)^{\frac{2}{r}}
    \le \frac{\eta}{c_1} \io \ue^{q-2}|\nabla \ue|^2
    + c_3
\]
in $(0, \Tmaxe)$ for any $\ep \in (0,1)$.
Together with \eqref{veL2Func:ineq1}, this readily
yields
\eqref{veL2Func:ddtveL2} with $C=c_1 c_3$.
\end{proof}

As a consequence of Lemmas~\ref{Lem:ueLqFunc1}
and \ref{Lem:veL2Func},
we can derive $L^2$ estimates for $\ve$
on the basis of the observation that under a
smallness condition on $p$, functionals of the form
$\sgn(q-1)\io \ue^q + C \io \ve^2$ enjoy entropy-like properties.

%
%
\begin{lem}\label{Lem:veL2}
Let $n \in \N$, and
suppose that $p > 1$ and $\theta > 1$ satisfy
\begin{align}\label{veL2:p}
    \begin{cases}
        p \in \left(1, 
          \min\left\{2, 
              \dfrac{2\theta+1}{2\theta-1}\right\}\right)
          & \quad\mbox{if}\ n=1, \\[3mm]
        p \in \left(1, \dfrac{n\theta}{n\theta-1}\right)
          & \quad\mbox{if}\ n \ge 2.
    \end{cases}
\end{align}
Then there exists a constant $C>0$ such that
\begin{align}\label{veL2:veL2}
    \io \ve^2(\cdot,t) \le C
\end{align}
for all $t \in (0, \Tmaxe)$ and $\ep \in (0,1)$.
\end{lem}
\begin{proof}
We first claim that there exists $q > 0$ fulfilling both
\eqref{ueLqFunc1:q} and \eqref{veL2Func:q}.
When $n=1$ and $\theta \le \frac{3}{2}$, we only need to choose
$q \in (0, \frac{2(2-p)}{p-1})\setminus\{1\}$ to verify the claim.
If $n = 1$ and $\theta > \frac{3}{2}$, we see from the restriction
$p < \frac{2\theta+1}{2\theta-1}$ in
\eqref{veL2:p} that
$\frac{2(2-p)}{p-1} = \frac{2}{p-1} - 2 > 2\theta - 3$,
which guarantees the existence of such a constant $q$.
Similarly, when $n \ge 2$, the assumption
\eqref{veL2:p}
entails that $\frac{2(2-p)}{p-1} > 2\theta - \frac{4}{n}$ and thus
proves the claim.
Keeping this selection fixed,
we infer from the fact $\theta > 1$ that
$\frac{n\theta}{n\theta-1} < 1+\frac{1}{n-1} \le 1+\frac{2}{n}$ if
$n\ge 2$, so that
Lemma~\ref{Lem:ueLqFunc1} becomes applicable
so as to yield $c_1 > 0$ fulfilling
\begin{align}\label{veL2:ddtueLq}
    \sgn(q-1)\ddt \io \ue^q
    + \frac{q|q-1|}{4}\io \ue^{q-2}|\nabla \ue|^2
    \le c_1 \io |\nabla\ve|^2 + c_1
\end{align}
in $(0, \Tmaxe)$ for any $\ep\in (0,1)$.
Moreover, by Lemma~\ref{Lem:veL2Func},
there exists a constant $c_2 > 0$ such that
\[
    (c_1 + 1) \ddt\io \ve^2 + (c_1 + 1)\io \ve^2
    + (c_1 + 1)\io |\nabla \ve|^2
    \le \frac{q|q-1|}{8} \io \ue^{q-2}|\nabla\ue|^2
    + c_2
\]
in $(0, \Tmaxe)$ for all $\ep \in (0,1)$, whence
combining this with \eqref{veL2:ddtueLq} provides $c_3 > 0$
such that
\begin{align}
\nonumber
    &\ddt\left[\sgn(q-1)\io\ue^q + (c_1+1) \io \ve^2\right]
\\ \nonumber
    &\quad\, + \frac{q|q-1|}{8} \io \ue^{q-2}|\nabla\ue|^2
    + (c_1 + 1) \io \ve^2 + \io |\nabla \ve|^2
\\ 
    &\quad\, \le c_3
    \label{veL2:entropy1}
\end{align}
in $(0, \Tmaxe)$ for all $\ep \in (0,1)$.
Now if $q < 1$, by means of the Poincar\'{e}--Wirtinger inequality,
the H\"{o}lder inequality, \eqref{local_App:mass} and
\eqref{bound:uie}, we find $c_3, c_4 > 0$ such that
\begin{align}
\nonumber
    \io \ue^{q-2} |\nabla\ue|^2
    &\ge c_3 \io \ue^q - \frac{c_3}{|\Omega|}
      \left(\io \ue^{\frac{q}{2}}\right)^2
\\ \nonumber
    &\ge c_3 \io \ue^q - c_3 |\Omega|^{1-q} \left(\io \ue\right)^q
\\
    &\ge c_3 \io \ue^q - c_4
    \label{veL2:entropy2}
\end{align}
in $(0, \Tmaxe)$ for every $\ep \in (0,1)$.
On the other hand, when $q > 1$,
we invoke the Gagliardo--Nirenberg inequality
(Lemma~\ref{Lem:GN1}~\eqref{GN1-1}) along with the Young inequality,
\eqref{local_App:mass} and \eqref{bound:uie} to obtain $c_5 > 0$
such that
\begin{align}\label{veL2:entropy3}
    \io \ue^{q-2}|\nabla \ue|^2 \ge c_3 \io \ue^q - c_5
\end{align}
in $(0,\Tmaxe)$ for all $\ep \in (0,1)$.
Therefore, if we define
\[
    \mathcal{F}_{1,\ep}(t) := \sgn(q-1) \io \ue^q(\cdot,t)
      + (c_1+1) \io \ve^2(\cdot,t)
    \quad\mbox{for}\ t \in (0,\Tmaxe)\ \mbox{and}\ 
      \ep \in (0,1)
\]
and let $c_6 := \min\{\frac{q|q-1|}{8}c_3, 1\} > 0$,
it follows from
\eqref{veL2:entropy1}, \eqref{veL2:entropy2}
and \eqref{veL2:entropy3} that there exists $c_7 > 0$ satisfying
\[
    \ddt\mathcal{F}_{1,\ep} + c_6 \mathcal{F}_{1,\ep}
    + \io |\nabla \ve|^2 \le c_7
\]
in $(0, \Tmaxe)$ for all $\ep \in (0,1)$.
In light of \eqref{bound:uie} and \eqref{bound:vie},
this thereby entails \eqref{veL2:veL2}.
\end{proof}

We now go back to perform a standard $L^q$ testing
argument to the first equation in \eqref{Sys:App}
to assert bounds on $\ue$ in $L^q$ spaces for
any $q > 1$.
Unlike in Lemma~\ref{Lem:ueLqFunc1},
we will rely on the quantity
$\|\nabla \ve\|_{L^r(\Omega)}^{\kappa}$
with some $r, \kappa > 0$
instead of $\io |\nabla\ve|^2$,
which enables us to choose $q > 1$ arbitrary.

%
%
\begin{lem}\label{Lem:ueLqFunc2}
Let $n\in\N$,
$p > 1$ and $q > 1$, and suppose that $r > 0$ fulfills
\begin{align}\label{ueLqFunc2:r}
    r > \max\{2, n\}(p-1).
\end{align}
Then there exists a constant $C > 0$ such that
\begin{align}\label{ueLqFunc2:ddtueLq}
    \ddt \io\ue^q(\cdot,t) 
    + \frac{q(q-1)}{4} \io \ue^{q-2}(\cdot,t)|\nabla \ue(\cdot,t)|^2
    \le C \|\nabla \ve(\cdot,t)\|_{L^r(\Omega)}^{2a^{\ast}(p-1)} + C
\end{align}
for all $t \in (0, \Tmaxe)$ and $\ep \in (0,1)$, where
\begin{align}\label{ueLqFunc2:aStar}
    a^{\ast} := \frac{r(nq+2-n)}{2(r-n(p-1))} > 1.
\end{align}
\end{lem}
\begin{proof}
Similarly as in Lemma~\ref{Lem:ueLqFunc1},
we test the first equation in \eqref{Sys:App} by $\ue^{q-1}$ and
employ the Young inequality to see that
\begin{align}\label{ueLqFunc2:ineq2}
    \frac{1}{q} \ddt \io \ue^q 
    + \frac{q-1}{2} \io \ue^{q-2} |\nabla\ue|^2
    \le \frac{\chi^2(q-1)}{2} \io \ue^q |\nabla \ve|^{2(p-1)}
\end{align}
in $(0, \Tmaxe)$ for all $\ep \in (0,1)$.
Moreover,
since $r > 2(p-1)$, the H\"{o}lder inequality warrants that
\begin{align}
\nonumber
    \io \ue^q |\nabla \ve|^{2(p-1)}
    &\le \left(\io |\nabla \ve|^r\right)^{\frac{2(p-1)}{r}}
      \left(\io \ue^{\frac{qr}{r-2(p-1)}}\right)^{
        \frac{r-2(p-1)}{r}}
\\
    &= \|\nabla \ve\|_{L^r(\Omega)}^{2(p-1)}
      \big\|\ue^{\frac{q}{2}}\big\|_{
        L^{\frac{2r}{r-2(p-1)}}(\Omega)}^2
        \label{ueLqFunc2:ineq1}
\end{align}
in $(0, \Tmaxe)$ for each $\ep \in (0,1)$.
Now thanks to the fact that $r > n(p-1)$,
we can apply the Gagliardo--Nirenberg inequality
(Lemma~\ref{Lem:GN1}~\eqref{GN1-1}) to obtain $c_1 > 0$ such that
\[
    \big\|\ue^{\frac{q}{2}}\big\|_{
        L^{\frac{2r}{r-2(p-1)}}(\Omega)}^2
    \le c_1 \left(\io \ue^{q-2}|\nabla \ue|^2\right)^{a_1}
     \|\ue\|_{L^1(\Omega)}^{q(1-a_1)}
     + c_1 \|\ue\|_{L^1(\Omega)}^{q}
\]
in $(0, \Tmaxe)$ for all $\ep \in (0,1)$, where
$a_1 := \frac{\frac{q}{2}-\frac{r-2(p-1)}{2r}}
{\frac{q}{2}+\frac{1}{n}-\frac{1}{2}} \in (0,1)$.
In conjunction with \eqref{ueLqFunc2:ineq1},
\eqref{local_App:mass}, \eqref{bound:uie}
and the Young inequality,
this entails the existence of $c_2, c_3 > 0$ such that
\begin{align}
\nonumber
    \io \ue^q |\nabla \ve|^{2(p-1)}
    &\le \frac{1}{2\chi^2} \io \ue^{q-2} |\nabla\ue|^2
    + c_2 \|\nabla \ve\|_{L^r(\Omega)}^{\frac{2(p-1)}{1-a_1}}
    + c_2 \|\nabla \ve\|_{L^r(\Omega)}^{2(p-1)}
\\
    &\le \frac{1}{2\chi^2} \io \ue^{q-2} |\nabla\ue|^2
    + c_3 \|\nabla \ve\|_{L^r(\Omega)}^{\frac{2(p-1)}{1-a_1}}
    + c_3
    \label{ueLqFunc2:ineq3}
\end{align}
in $(0, \Tmaxe)$ for all $\ep \in (0,1)$.
Since
$1-a_1 = \frac{\frac{1}{n}-\frac{p-1}{r}}
{\frac{q}{2}+\frac{1}{n}-\frac{1}{2}}
=\frac{2(r-n(p-1))}{r(nq+2-n)}$ holds,
a combination of \eqref{ueLqFunc2:ineq2} and
\eqref{ueLqFunc2:ineq3} thus proves \eqref{ueLqFunc2:ddtueLq}
with $C = \frac{\chi^2 q (q-1)}{2} c_3$.
\end{proof}

%
%
\begin{remark}\label{Rem:r2}
Let us emphasize that if $p>1$ and $\theta > 1$ are as in
the assumption \eqref{veL2:p} in
Lemma~\ref{Lem:veL2},
then $r = 2$ actually satisfies \eqref{ueLqFunc2:r}.
Indeed, this is obvious when $n=1$, because
$2(p-1)<2$.
If $n\ge 2$, we observe that the condition
$p < \frac{n\theta}{n\theta-1}$ with $\theta>1$ warrants that
$n(p-1) < \frac{n}{n\theta-1}<\frac{n}{n-1}\le 2$.

In particular, by Lemma~\ref{Lem:ueLqFunc2} we know that
if $p > 1$ and $\theta > 1$ satisfy
\eqref{veL2:p},
then for any $q > 1$ there exists a constant $C > 0$ such that
\begin{align}
\notag
    &\ddt \io \ue^q(\cdot,t)
    + \frac{q(q-1)}{4}\io \ue^{q-2}(\cdot,t)|\nabla \ue(\cdot,t)|^2
\\
    &\quad\, \le C\left(\io |\nabla \ve(\cdot,t)|^2\right)^{
      \frac{(nq+2-n)(p-1)}{2-n(p-1)}} + C
      \label{remark:ueLq}
\end{align}
for all $t \in (0,\Tmaxe)$ and $\ep \in (0,1)$.
\end{remark}

In light of \eqref{remark:ueLq},
to establish $L^q$ bounds for $\ue$
it will be important to derive
$L^2$ estimates for $\nabla \ve$.
To achieve this we begin with an essential
differential inequality for the function $\io |\nabla \ve|^2$
in the following lemma.

%
%
\begin{lem}\label{Lem:vexL2Func}
Let $n\in\N$, and suppose that
$\theta > 1$ and $q > 1$ are such that
\begin{align}\label{vexL2Func:q}
    q > 2\theta - \frac{2}{n}
\end{align}
holds.
Then for any $\eta > 0$ there exists a constant $C>0$ fulfilling
\begin{align}
\nonumber
    &\ddt\io|\nabla \ve(\cdot,t)|^2
    +2\io |\nabla \ve(\cdot,t)|^2
    + \io |\Delta \ve(\cdot,t)|^2
\\
    &\quad\, \le \eta \io \ue^{q-2}(\cdot,t)|\nabla \ue(\cdot,t)|^2
    + C
    \label{vexL2Func:ddtvexL2}
\end{align}
for all $t \in (0,\Tmaxe)$ and $\ep \in (0,1)$.
\end{lem}
\begin{proof}
Testing the second equation in \eqref{Sys:App} against
$-\Delta \ve$ with an application of the Young inequality
guarantees that
\begin{align}
\nonumber
    \frac{1}{2} \ddt \io |\nabla\ve|^2
    + \io |\Delta \ve|^2 + \io |\nabla\ve|^2
    &= - \io \ue^{\theta} \Delta\ve
\\
    &\le \frac{1}{2} \io |\Delta\ve|^2 + \frac{1}{2}\io \ue^{2\theta}
    \label{vexL2Func:ineq1}
\end{align}
in $(0,\Tmaxe)$ for all $\ep \in (0,1)$,
where the Gagliardo--Nirenberg inequality
(Lemma~\ref{Lem:GN1}~\eqref{GN1-1}) enables us to pick $c_1 > 0$ satisfying
\begin{align}
\nonumber
    \io \ue^{2\theta} &= \big\|\ue^{\frac{q}{2}}\big\|_{
      L^{\frac{4\theta}{q}}(\Omega)}^{
      \frac{4\theta}{q}}
\\
    &\le c_1 \left(\io \ue^{q-2}|\nabla\ue|^2\right)^{
      \frac{2\theta a_1}{q}}
      \|\ue\|_{L^1(\Omega)}^{2\theta(1-a_1)}
    + c_1 \|\ue\|_{L^1(\Omega)}^{2\theta}
    \label{vexL2Func:ineq2}
\end{align}
in $(0, \Tmaxe)$ for all $\ep \in (0,1)$
with $a_1 := \frac{\frac{q}{2}-\frac{q}{4\theta}}
{\frac{q}{2}+\frac{1}{n}-\frac{1}{2}}\in(0,1)$.
Since the hypothesis \eqref{vexL2Func:q} asserts that
\[
    \frac{2\theta a_1}{q} 
    = \frac{\theta-\frac{1}{2}}{\frac{q}{2}+\frac{1}{n}-\frac{1}{2}}
    < \frac{\theta-\frac{1}{2}}
    {\frac{1}{2}(2\theta-\frac{2}{n})+\frac{1}{n}-\frac{1}{2}}
    = 1,
\]
through the Young inequality
a combination of \eqref{vexL2Func:ineq2} with
\eqref{local_App:mass} and \eqref{bound:uie} shows that
for each $\eta > 0$, one can find $c_2 > 0$ such that
\[
    \io \ue^{2\theta} \le \eta \io \ue^{q-2}|\nabla \ue|^2
    + c_2
\]
in $(0, \Tmaxe)$ for all $\ep \in (0,1)$.
Therefore, together with \eqref{vexL2Func:ineq1}
this confirms \eqref{vexL2Func:ddtvexL2} with $C=c_2$.
\end{proof}

The form of the first integral on the right of
\eqref{vexL2Func:ddtvexL2} suggests that
the coupled quantities
$\io \ue^q + C\io |\nabla\ve|^2$ with some $C>0$
play the role of entropy-like functionals,
provided that $q > 1$ lies within appropriate ranges.
In Lemma~\ref{Lem:ueLqFunc1} we already established
the differential inequality \eqref{ueLqFunc1:ddtueLq}
for $\io \ue^q$, however,
this does not contain the integral $\io |\Delta\ve|^2$,
which would be possible to improve.
As a key step in this direction, we shall focus on the
intermediate inequality \eqref{ueLqFunc2:ddtueLq}
already obtained in Lemma~\ref{Lem:ueLqFunc2},
and derive further estimates involving $\Delta \ve$
in the space $L^2(\Omega)$ for the term
$\|\nabla\ve\|_{L^r(\Omega)}^{\kappa}$.

%
%
\begin{lem}\label{Lem:vexLr}
Let $n \in \N$,
and suppose that $p > 1$ and $q > 1$ satisfy
\begin{align}\label{vexLr:pq}
    \begin{cases}
        p \in \left(1, \dfrac{n+6}{n+2}\right)
        \quad\mbox{and}\quad
        q \in \left(1, \dfrac{n^2 p - n^2-4p + 12}
          {n^2 p - n^2+2np-2n}\right)
        & \mbox{if}\ n = 1, 2,
    \\[3mm]
        p \in \left(1, \dfrac{n+2}{n}\right)
        \quad\mbox{and}\quad
        q \in \left(1, \dfrac{2}{n(p-1)}\right)
        & \mbox{if}\ n \ge 3.
    \end{cases}
\end{align}
Then one can find $r > 0$ fulfilling \eqref{ueLqFunc2:r} and
\begin{align}\label{vexLr:r}
    \begin{cases}
        r \in [2, \infty) & \mbox{if}\ n=1,2,
        \\[1mm]
        r \in \left[2, \dfrac{2n}{n-2}\right)
        & \mbox{if}\ n \ge 3
    \end{cases}
\end{align}
as well as
\begin{align}\label{vexLr:bStar}
    2a^{\ast} \cdot b_{\ast} (p-1) < 2,
\end{align}
where $a^{\ast} > 1$ is defined by \eqref{ueLqFunc2:aStar},
and where
\[
    b_{\ast} := \frac{\frac{1}{2}+\frac{1}{n}-\frac{1}{r}}
    {\frac{2}{n}} \in \left[\frac{1}{2}, 1\right).
\]
Furthermore, for all $\eta > 0$ there exists a constant $C>0$
such that
\begin{align}
\notag
    &
    \|\nabla\ve(\cdot,t)\|_{L^r(\Omega)}^{
      2a^{\ast}(p-1)}
\\
    &\quad\,
    \le\eta \io |\Delta \ve(\cdot,t)|^2
    + C\left(\io \ve^2(\cdot,t)\right)^{
      \frac{1}{2-2a^{\ast}\cdot b_{\ast}(p-1)}}
    + C\left(\io \ve^2(\cdot,t)\right)^{a^{\ast}(p-1)}
    \label{vexLr:vexLr}
\end{align}
for all $t \in (0, \Tmaxe)$ and $\ep \in (0,1)$.
\end{lem}
\begin{proof}
A straightforward computation implies that
\begin{gather*}
    \frac{(n+6)\xi + 2n}{(n+2)\xi + 2n} \upto
    \frac{n+6}{n+2}
\\ \intertext{and}
    \frac{(n^2 p-n^2-4p+12)\xi -2n^2 p+2n^2-4np+4n}
    {(n^2 p-n^2+2np-2n)\xi -2n^2 p+2n^2}
    \upto
    \frac{n^2 p - n^2-4p + 12}
          {n^2 p - n^2+2np-2n}
\end{gather*}
as $\xi \upto \infty$, and if furthermore $n\ge 3$, then
\begin{gather*}
    \frac{(n+6)\xi + 2n}{(n+2)\xi + 2n} \upto
    \frac{n+2}{n}
\\ \intertext{as well as}
    \frac{(n^2 p-n^2-4p+12)\xi -2n^2 p+2n^2-4np+4n}
    {(n^2 p-n^2+2np-2n)\xi -2n^2 p+2n^2}
    \upto \frac{2}{n(p-1)}
\end{gather*}
as $\xi \upto \frac{2n}{n-2}$.
These four properties would entail that
there exists $r \ge 2$ satisfying both \eqref{ueLqFunc2:r}
and \eqref{vexLr:r} which is such that
$p < \frac{(n+6)r+2n}{(n+2)r+2n}$ and
\begin{align}\label{vexLr:q}
    q < \frac{(n^2 p-n^2-4p+12)r -2n^2 p+2n^2-4np+4n}
    {(n^2 p-n^2+2np-2n)r -2n^2 p+2n^2}.
\end{align}
Keeping those numbers $p,q,r$, we claim that
$2a^{\ast} \cdot b_{\ast}(p-1) < 2$ holds.
Indeed, according to \eqref{vexLr:q} we have
\begin{align*}
    nq + 2 - n
    &< \frac{(-4p+12-2np+2n)rn-4n^2 p+4n^2}
      {(n^2 p-n^2+2np-2n)r - 2n^2 p+2n^2} + 2
\\
    &= \frac{8nr-8n^2 p +8n^2}
    {(n^2 p-n^2+2np-2n)r - 2n^2 p+2n^2}
\\
    &= \frac{8(r-n(p-1))}{(p-1)(nr+2r-2n)},
\end{align*}
and, as a consequence thereof, we observe that
\begin{align*}
    2a^{\ast} \cdot b_{\ast}(p-1)
    &= \frac{r(nq+n-2)}{r-n(p-1)} \cdot
      \frac{\frac{1}{2}+\frac{1}{n}-\frac{1}{r}}{\frac{2}{n}}
      (p-1)
\\
    &< \frac{8(r-n(p-1))}{(p-1)(nr+2r-2n)} \cdot
    \frac{r(p-1)}{r-n(p-1)}\cdot
    \frac{nr+2r-2n}{4nr}
\\
    &= \frac{8}{4n}
\\
    &\le 2,
\end{align*}
which yields the claim.
Here since \eqref{vexLr:r} warrants that
$b_{\ast} \in [\frac{1}{2}, 1)$, the Gagliardo--Nirenberg inequality
(Lemma~\ref{Lem:GN2}) becomes applicable so as to provide
$c_1 > 0$ such that
\begin{align}\label{vexLr:GN}
    \|\nabla\ve\|_{L^r(\Omega)}^{2a^{\ast}(p-1)}
    \le c_1 \|\Delta\ve\|_{L^2(\Omega)}^{
      2a^{\ast} \cdot\,b_{\ast}(p-1)}
    \|\ve\|_{L^2(\Omega)}^{
      2a^{\ast}(1-b_{\ast})(p-1)}
    + c_1 \|\ve\|_{L^2(\Omega)}^{
      2a^{\ast}(p-1)}
\end{align}
in $(0, \Tmaxe)$ for any $\ep\in (0,1)$.
Thanks to the aforementioned claim
$2a^{\ast}\cdot b_{\ast}(p-1) < 2$,
in \eqref{vexLr:GN} we can thus apply the Young inequality
to establish the desired inequality.
\end{proof}

Having thus asserted inequalities
\eqref{ueLqFunc2:ddtueLq},
\eqref{vexL2Func:ddtvexL2} and \eqref{vexLr:vexLr},
we rely on the functionals
$\io \ue^q + \io |\nabla\ve|^2$
to derive $L^2$ estimates for $\nabla\ve$
under the assumption in Lemma~\ref{Lem:veL2}.

%
%
\begin{lem}\label{Lem:vexL2}
Let $n\in \N$,
and let $p > 1$ and $\theta > 1$
be as in \eqref{veL2:p}.
Then there is a constant $C>0$ satisfying
\begin{align}\label{vexL2:vexL2}
    \io |\nabla\ve(\cdot,t)|^2 \le C
\end{align}
for any $t \in (0,\Tmaxe)$ and $\ep \in (0,1)$.
\end{lem}
\begin{proof}
We first claim that since $p>1$ satisfies \eqref{veL2:p},
it also fulfills the assumption \eqref{vexLr:pq}
in Lemma~\ref{Lem:vexLr}.
Indeed, this is obvious when $n=1,2$ since it follows that
$\frac{n+6}{n+2} \ge 2$, and
if $n \ge 3$, we see from
the fact $\theta > 1 > \frac{n+2}{2n}$ that
$\frac{n\theta}{n\theta-1} = 1+\frac{1}{n\theta-1}
<1+\frac{2}{n} = \frac{n+2}{n}$,
which yields the claim.
We next claim that there exists $q > 1$ fulfilling
\eqref{vexL2Func:q} and \eqref{vexLr:pq}
under the condition \eqref{veL2:p}
in Lemma~\ref{Lem:veL2}.
To verify this, if $n \ge 2$,
then we see from \eqref{veL2:p} that
$2\theta-\frac{2}{n} < \frac{2p}{n(p-1)}-\frac{2}{n}
=\frac{2}{n(p-1)}$,
which shows that the claim holds
even when $n=2$ since the identity $\frac{2}{n(p-1)}
= \frac{n^2 p -n^2-4p+12}{n^2 p-n^2+2np-2n}$
holds for $n=2$ and $p > 1$.
When $n=1$, we infer from the restriction
$p<\frac{2\theta+1}{2\theta-1}$ that
$2\theta-2 < \frac{3-p}{p-1} < \frac{11-3p}{3(p-1)}$,
and hence proves the claim.
Therefore, thanks to Lemma~\ref{Lem:vexLr},
we can choose $r \ge 2$ satisfying
\eqref{ueLqFunc2:r}, \eqref{vexLr:r} as well as
\eqref{vexLr:bStar},
whence from Lemma~\ref{Lem:ueLqFunc2} one can find
a constant $c_1 > 0$ such that
\begin{align}\label{vexL2:ineq1}
    \ddt\io\ue^q + \frac{q(q-1)}{4}\io\ue^{q-2}|\nabla\ue|^2
    \le c_1 \|\nabla\ve\|_{L^r(\Omega)}^{
      2a^{\ast}(p-1)} + c_1
\end{align}
in $(0, \Tmaxe)$ for all $\ep \in (0,1)$,
where $a^{\ast} > 1$ is taken from \eqref{ueLqFunc2:aStar}.
Again by Lemma~\ref{Lem:vexLr} along with
Lemma~\ref{Lem:veL2}, we can moreover
pick $c_2 > 0$ such that
\begin{align}\label{vexL2:ineq2}
    \|\nabla\ve\|_{L^r(\Omega)}^{2a^{\ast}(p-1)}
    \le \frac{1}{2c_1} \io |\Delta \ve|^2 + c_2
\end{align}
in $(0,\Tmaxe)$ for any $\ep\in (0,1)$.
In addition, we employ Lemma~\ref{Lem:vexL2Func} to obtain
$c_3 > 0$ such that
\begin{align}\label{vexL2:ineq3}
    \ddt\io |\nabla\ve|^2 + 2\io|\nabla\ve|^2
    + \io|\Delta\ve|^2
    \le \frac{q(q-1)}{8}\io\ue^{q-2}|\nabla\ue|^2 + c_3
\end{align}
in $(0, \Tmaxe)$ for all $\ep \in (0,1)$.
Consequently, a combination of \eqref{vexL2:ineq1},
\eqref{vexL2:ineq2} and \eqref{vexL2:ineq3} would
yield $c_4 > 0$ fulfilling
\begin{align}\label{vexL2:ineq4}
    \ddt\left[\io\ue^q+\io|\nabla\ve|^2\right]
    +\frac{q(q-1)}{8}\io\ue^{q-2}|\nabla\ue|^2
    +2\io|\nabla\ve|^2+\frac{1}{2}\io|\Delta\ve|^2
    \le c_4
\end{align}
in $(0, \Tmaxe)$ for all $\ep \in (0,1)$.
On the other hand, similarly as in the proof of
Lemma~\ref{Lem:veL2} we make use of
the Gagliardo--Nirenberg inequality
(Lemma~\ref{Lem:GN1}~\eqref{GN1-1}) and the Young inequality along with
\eqref{local_App:mass} as well as \eqref{bound:uie} to find
$c_5 > 0$ such that
\[
    \io \ue^{q-2}|\nabla\ue|^2 \ge \frac{16}{q(q-1)}\io \ue^q
    - c_5
\]
in $(0, \Tmaxe)$ for each $\ep \in (0,1)$.
In conclusion,
inserting this into \eqref{vexL2:ineq4} we infer that
there exists $c_6 > 0$ with the property that
the function
\[
    \mathcal{F}_{2,\ep}(t) := \io \ue^q(\cdot,t)
    + \io |\nabla\ve(\cdot,t)|^2
    \quad\mbox{for}\ t \in (0,\Tmaxe)\ \mbox{and}\ 
    \ep \in(0,1)
\]
satisfies
\[
    \ddt \mathcal{F}_{2,\ep} + 2\mathcal{F}_{2,\ep}
    + \frac{1}{2} \io |\Delta\ve|^2 \le c_6
\]
in $(0, \Tmaxe)$ for all $\ep \in (0,1)$,
whence in light of \eqref{bound:vie} this completes the proof.
\end{proof}

Thus, knowing that the intermediate inequality
\eqref{ueLqFunc2:ddtueLq} holds true with $r=2$
under the hypothesis \eqref{veL2:p} by Remark~\ref{Rem:r2},
with the aid of \eqref{vexL2:vexL2}
we can apply a standard testing procedure to derive
$L^{q}$ estimates for $\ue$ which arbitrary $q > 1$.

%
%
\begin{lem}\label{Lem:ueLq}
Let $n\in \N$,
and suppose that $p>1$ and $\theta > 1$ satisfy \eqref{veL2:p}.
Then for any $q > 1$,
there exists a constant $C > 0$ such that
\begin{align}
    &\io \ue^q(\cdot,t) \le C
    \label{ueLq:ueLq}
\\ \intertext{and}
    &\int_{(t-1)_{+}}^t \io \ue^{q-2}|\nabla\ue|^2 \le C
    \label{ueLq:grad}
\end{align}
for all $t \in (0,\Tmaxe)$ and $\ep \in(0,1)$, where
$(t-1)_{+} := \max\{0, t-1\}$.
\end{lem}
\begin{proof}
Fixing $q > 1$ arbitrary, we infer from Remark~\ref{Rem:r2} that
there exists a constant $c_1 > 0$ fulfilling
\begin{align}\label{ueLq:ineq1}
    \ddt\io\ue^q + \frac{q(q-1)}{4}\io\ue^{q-2}|\nabla\ue|^2
    \le c_1 \left(\io|\nabla\ve|^2\right)^{
      \frac{(nq+2-n)(p-1)}{2-n(p-1)}} + c_1
\end{align}
in $(0,\Tmaxe)$ for all $\ep \in (0,1)$.
Since thanks to Lemma~\ref{Lem:vexL2}
we find $c_2 > 0$ such that
\begin{align}\label{ueLq:vexL2}
    \io |\nabla\ve|^2 \le c_2
\end{align}
in $(0,\Tmaxe)$ for each $\ep \in (0,1)$,
and since by adaptation of the reasoning
from Lemma~\ref{Lem:vexL2}
we obtain $c_3 > 0$ such that
\[
    \io \ue^{q-2}|\nabla\ue|^2 \ge \io \ue^{q} - c_3
\]
in $(0, \Tmaxe)$ for all $\ep \in (0,1)$,
we combine this and \eqref{ueLq:vexL2} with \eqref{ueLq:ineq1}
to find $c_4 > 0$ such that
\[
    \ddt\io\ue^q + \frac{q(q-1)}{4} \io\ue^{q-2}|\nabla\ue|^2
    \le c_4
\]
in $(0, \Tmaxe)$ for all $\ep \in (0,1)$.
In light of \eqref{bound:uie}, this establishes
\eqref{ueLq:ueLq} and \eqref{ueLq:grad}.
\end{proof}

As a final preparation,
in the following lemma we observe that
$\nabla \ve$ have $\ep$-independent bounds
in $L^{s}(\Omega)$ space, which will not only
enable us to pass to the limit as $\ep \dwto 0$ in the
second equation in \eqref{Sys:App}, but also
benefit our derivation of global boundedness features
in Lemma~\ref{Lem:ueLinfty}.

%
%
\begin{lem}\label{Lem:vexLs}
Let $n \in \N$,
and assume that $p > 1$ and $\theta > 1$ are as in \eqref{veL2:p}.
Then there exists a constant $C> 0$ such that
\begin{align}\label{vexLs:vexLs}
    \io |\nabla \ve(\cdot,t)|^s \le C
\end{align}
for any $t \in (0, \Tmaxe)$ and $\ep \in (0,1)$,
where $s \ge 2$ is as determined by \eqref{Main:s}.
\end{lem}
\begin{proof}
If $s=2$, then we already obtained \eqref{vexLs:vexLs} in
Lemma~\ref{Lem:vexL2}.
We henceforth consider the case $s > 2$.

From the second equation in \eqref{Sys:App} it follows that
for each $\ep \in (0,1)$,
\[
    \nabla (\ve)_t = \nabla \Delta \ve - \nabla \ve
    + \nabla \ue^{\theta}
\]
holds in $\Omega \times (0, \Tmaxe)$.
We test this by $|\nabla \ve|^{s-2}\nabla \ve$,
and then upon several integration by parts
along with the identity
$\nabla \ve \cdot \nabla\Delta\ve
=\frac{1}{2}\Delta |\nabla\ve|^2 - |D^2 \ve|^2$
to see that
\begin{align}
\notag
    &\frac{1}{s} \ddt \io |\nabla\ve|^s
\\ \notag
    &\quad\, = \frac{1}{2} \io |\nabla \ve|^{s-2} \Delta|\nabla\ve|^2
    - \io |\nabla\ve|^{s-2} |D^2 \ve|^2
    - \io |\nabla\ve|^s
    + \io |\nabla\ve|^{s-2} \nabla \ve \cdot \nabla \ue^{\theta}
\\ \notag
    &\quad\, = \frac{1}{2} \int_{\pa\Omega} 
      |\nabla\ve|^{s-2} (\nabla |\nabla\ve|^2\cdot \nu)
    - \frac{s-2}{4}
      \io |\nabla\ve|^{s-4} \big|\nabla|\nabla\ve|^2\big|^2
    - \io |\nabla\ve|^{s-2} |D^2 \ve|^2
\\
    &\qquad\, - \io |\nabla\ve|^s
    - \io \ue^{\theta} |\nabla\ve|^{s-2} \Delta\ve
    - \frac{s-2}{2}
      \io \ue^{\theta} |\nabla\ve|^{s-4}
        \nabla\ve\cdot\nabla|\nabla\ve|^2
    \label{vexLs:ineq1}
\end{align}
in $(0, \Tmaxe)$ for all $\ep \in (0,1)$.
To estimate the boundary integral on the right-hand side
we apply \cite[Lemma 4.2]{MS-2014} and rely on
the argument
based on the fractional Gagliardo--Nirenberg inequality
(cf.\ \cite[Theorem 1]{BM-2019})
in
\cite[Lemma 2.1~c)]{LW-2017}
to obtain $c_1 > 0$ fulfilling
\begin{align}\label{vexLs:boundary}
    \int_{\pa\Omega} |\nabla\ve|^{s-2}
      (\nabla|\nabla\ve|^2\cdot\nu)
    \le \frac{s-2}{4}
      \io |\nabla\ve|^{s-4}\big|\nabla|\nabla\ve|^2\big|^2
    + c_1 \left(\io |\nabla\ve|^2\right)^{\frac{s}{2}}
\end{align}
in $(0, \Tmaxe)$ for all $\ep \in (0,1)$.
On the other hand, according to
the pointwise estimate $|\Delta \ve| \le \sqrt{n} |D^2 \ve|$
(see \cite[Lemma 2.1~a)]{LW-2017})
and the identity
$\nabla|\nabla\ve|^2 = 2D^2 \ve \nabla \ve$
for $\ep \in (0,1)$,
we twice employ the Young inequality
to find $c_2 > 0$ such that
\begin{align*}
    &- \io \ue^{\theta} |\nabla\ve|^{s-2} \Delta\ve
    - \frac{s-2}{2}
      \io \ue^{\theta} |\nabla\ve|^{s-4}
        \nabla\ve\cdot\nabla|\nabla\ve|^2
\\
    &\quad\, \le (s-2+\sqrt{n})
      \io |\nabla\ve|^{s-2} |D^2 \ve| \ue^{\theta}
\\
    &\quad\, \le \frac{1}{2} \io |\nabla\ve|^{s-2} |D^2\ve|^2
    + \frac{(s-2+\sqrt{n})^2}{2} \io |\nabla\ve|^{s-2} \ue^{2\theta}
\\
    &\quad\, \le \frac{1}{2} \io |\nabla\ve|^{s-2} |D^2\ve|^2
    + \frac{1}{2} \io |\nabla\ve|^s + c_2 \io \ue^{\theta s}
\end{align*}
in $(0, \Tmaxe)$ for any $\ep \in (0,1)$.
Thus, inserting this and \eqref{vexLs:boundary} into
\eqref{vexLs:ineq1} provides $c_3 > 0$ such that
\begin{align}
\nonumber
    &\ddt \io |\nabla\ve|^s + \frac{s(s-2)}{8}
      \io |\nabla\ve|^{s-4}\big|\nabla|\nabla\ve|^2\big|^2
      + \frac{s}{2} \io |\nabla\ve|^{s-2} |D^2\ve|^2
      + \frac{s}{2} \io |\nabla\ve|^s
\\
      &\quad\, \le c_3 \left(\io|\nabla\ve|^2\right)^{\frac{s}{2}}
      + c_3 \io \ue^{\theta s}
      \label{vexLs:ineq2}
\end{align}
in $(0, \Tmaxe)$ for all $\ep \in (0,1)$.
Since Lemmas~\ref{Lem:vexL2} and \ref{Lem:ueLq} ensure
the existence of $c_4 > 0$ such that
\[
    \io |\nabla\ve|^2 \le c_4
    \quad\mbox{and}\quad
    \io \ue^{\theta s} \le c_4
\]
in $(0, \Tmaxe)$ for all $\ep \in (0,1)$,
the inequality \eqref{vexLs:ineq2}
together with \eqref{bound:vie}
thereby establishes \eqref{vexLs:vexLs}.
\end{proof}

We are now in a position to
derive $L^{\infty}$ estimates for $\ue$
and to verify that approximate solutions can exist globally,
provided that $p > 1$ and $\theta > 1$ fulfill the hypothesis
\eqref{veL2:p} in Lemma~\ref{Lem:veL2}.

%
%
\begin{lem}\label{Lem:ueLinfty}
Let $n \in \N$ be arbitrary,
and suppose that $p > 1$ and $\theta > 1$ satisfy \eqref{veL2:p}.
Then there exists a constant $C > 0$ such that
\begin{align}\label{ueLinfty:sum}
    \|\ue(\cdot,t)\|_{L^{\infty}(\Omega)}
    + \|\ve(\cdot,t)\|_{W^{1,s}(\Omega)}
    \le C
\end{align}
for all $t \in (0, \Tmaxe)$ and $\ep \in (0,1)$.

In particular, for each $\ep \in (0,1)$ the pair of functions
$(\ue,\ve)$ is a global classical solution of \eqref{Sys:App},
that is, $\Tmaxe= \infty$.
\end{lem}
\begin{proof}
Thanks to Lemmas~\ref{Lem:ueLq} and \ref{Lem:vexLs}
as well as the assumption $s > (n+2)(p-1)$ in \eqref{Main:s},
we find a number $q > n+2$ and a constant $c_1 > 0$
such that the function
$F_{\ep}:=\ue (|\nabla\ve|^2+\ep)^{\frac{p-2}{2}} \nabla\ve$
belongs to $L^{\infty}(0, \Tmaxe; L^q(\Omega)^{n})$
and that
$\|F_{\ep}(t)\|_{L^q(\Omega)} \le c_1$ for all
$t \in (0, \Tmaxe)$ and $\ep \in (0,1)$.
In light of \eqref{bound:uie}, this allows us to perform
the Moser-type iteration argument in \cite[Lemma A.1]{TW-2012},
so that we obtain $c_2 > 0$ satisfying
\begin{align}\label{ueLinfty:Linfty}
    \|\ue(\cdot,t)\|_{L^{\infty}(\Omega)}
    \le c_2
\end{align}
for all $t \in (0,\Tmaxe)$ and $\ep \in (0,1)$.
Apart from that, once more relying on Lemma~\ref{Lem:vexLs},
and invoking the Poincar\'{e}--Wirtinger inequality,
the Young inequality and Lemma~\ref{Lem:veL2}, we
can find $c_3, c_4 > 0$ such that
\begin{align*}
    \|\ve(\cdot,t)\|_{L^s(\Omega)}
    &\le \left\|\ve(\cdot,t)-\frac{1}{|\Omega|}\io\ve(\cdot,t)
      \right\|_{L^s(\Omega)}
    + |\Omega|^{\frac{1}{s} - 1} \io\ve(\cdot,t)
\\
    &\le c_3 \|\nabla \ve(\cdot,t)\|_{L^s(\Omega)}
    + \frac{1}{2} \io \ve^2(\cdot,t) 
    + \frac{1}{2} |\Omega|^{\frac{2}{s}-1}
\\
    &\le c_4
\end{align*}
for all $t \in (0, \Tmaxe)$ and $\ep \in (0,1)$.
Together with \eqref{vexLs:vexLs} and \eqref{ueLinfty:Linfty},
this establishes \eqref{ueLinfty:sum}.
In view of \eqref{local_App:criterion}, the latter statement
being valid due to \eqref{ueLinfty:sum}.
\end{proof}

As a last ingredient for our analysis, let us add
a boundedness statement
within the parameter range described in \eqref{Main_Assumption}
in the case $n=1$,
which is not fully covered by the previous lemmas.
In this situation, standard estimates
for the Neumann heat semigroup become applicable and
play a crucial role.

%
%
\begin{lem}\label{Lem:1D}
Let $n=1$, and suppose that $p > 1$ and $\theta > 1$
fulfill
\begin{align}\label{1D:p}
    p \in \left[\min\left\{2, \frac{2\theta+1}{2\theta-1}\right\},
    \frac{\theta}{\theta-1}\right).
\end{align}
Then one can find a constant $C> 0$ such that
\begin{align}\label{1D:sum}
    \|\ue(\cdot,t)\|_{L^{\infty}(\Omega)}
    + \|\ve(\cdot,t)\|_{W^{1,\infty}(\Omega)} \le C
\end{align}
for all $t \in (0,\Tmaxe)$ and $\ep \in (0,1)$.

In particular, for every $\ep \in (0,1)$ the pair of functions
$(\ue,\ve)$ is a global classical solution of \eqref{Sys:App},
namely, $\Tmaxe= \infty$.
\end{lem}
\begin{proof}
We first note that since \eqref{1D:p} ensures
that
$(\theta-1)(p-1)+(1-1)<1$,
there exist $q, r > 1$ such that
\begin{align}\label{1D:crucial}
    \left(\theta-\frac{1}{r}\right)(p-1)
    + \left(1-\frac{1}{q}\right) < 1
\end{align}
is satisfied.
Keeping this parameter $q,r>1$ fixed,
on the basis of a variation-of-constants representation
associated with \eqref{Sys:App}
we apply Lemma~\ref{Lem:semigroup}~\eqref{semigroup-1},
\eqref{semigroup-2},
\eqref{local_App:mass}
and \eqref{bound:uie} to find $c_1, c_2 > 0$ such that
\begin{align}
\nonumber
    &\|\nabla\ve(\cdot,t)\|_{L^{\infty}(\Omega)}
\\ \nonumber
    &\quad\,\le c_1 \|\nabla v_0\|_{L^{\infty}(\Omega)}
    + c_1 \int_0^t (1+(t-\sigma)^{-\frac{1}{2}-\frac{1}{2r}})
    e^{-(1+\lambda_1)(t-\sigma)} \|\ue^{\theta}(\cdot,\sigma)\|_{
      L^r(\Omega)}\, \intd{\sigma}
\\ \nonumber
    &\quad\,\le c_1 \|v_0\|_{W^{1,\infty}(\Omega)}
\\ \nonumber
    &\qquad\, + c_1 \int_0^t (1+(t-\sigma)^{-\frac{1}{2}-\frac{1}{2r}})
    e^{-(1+\lambda_1)(t-\sigma)}
    \|\ue(\cdot,\sigma)\|_{L^1(\Omega)}^{\frac{1}{r}}
    \|\ue(\cdot,\sigma)\|_{L^{\infty}(\Omega)}^{\theta-\frac{1}{r}}
    \, \intd{\sigma}
\\
    &\quad\, \le c_1 \|v_0\|_{W^{1,\infty}(\Omega)}
    + c_2 c_3 \left(\sup_{\sigma \in (0, \Tmaxe)}
      \|\ue(\cdot, \sigma)\|_{L^{\infty}(\Omega)}\right)^{
        \theta-\frac{1}{r}}
    \label{1D:1}
\end{align}
for all $t \in (0,\Tmaxe)$ and $\ep \in (0,1)$, where
$c_3 := \int_0^{\infty} (1+z^{-\frac{1}{2}-\frac{1}{2r}})e^{
-(1+\lambda_1)z}\,\intd{z}$ is finite due to the fact that $r>1$.
According to \eqref{1D:1}, we rely on
Lemma~\ref{Lem:semigroup}~\eqref{semigroup-3}
and again make use of \eqref{local_App:mass}
as well as \eqref{bound:uie} to obtain $c_5, c_6 >0$ such that
\begin{align}
\nonumber
    &\|\ue(\cdot,t)\|_{L^{\infty}(\Omega)}
\\ \nonumber
    &\quad\, \le \|\uie\|_{L^{\infty}(\Omega)}
    + c_4\int_0^t (1+(t-\sigma)^{-\frac{1}{2}-\frac{1}{2q}})
    e^{-\lambda_1 (t-\sigma)}
    \big\|\ue(\cdot,\sigma)
      (|\nabla\ve(\cdot,\sigma)|^2+1)^{\frac{p-1}{2}}
    \big\|_{L^q(\Omega)}\, \intd{\sigma}
\\
    &\quad\,\le c_5 \|u_0\|_{L^{\infty}(\Omega)}
    + c_6 c_7 \left(\sup_{\sigma \in (0, \Tmaxe)}
      \|\ue(\cdot, \sigma)\|_{L^{\infty}(\Omega)}\right)^{
        (\theta-\frac{1}{r})(p-1)+(1-\frac{1}{q})}
        \label{1D:2}
\end{align}
for all $t \in (0,\Tmaxe)$ and $\ep \in (0,1)$, where
$c_7 := \int_0^{\infty} (1+z^{-\frac{1}{2}-\frac{1}{2q}})e^{
-\lambda_1 z}\,\intd{z}$ is finite since $q>1$.
Therefore, abbreviating
\[
    R_{\ep} := \sup_{t \in (0, \Tmaxe)} \|\ue(\cdot,t)\|_{
      L^{\infty}(\Omega)}
\]
for each $\ep \in (0,1)$, we see from \eqref{1D:2} that
\[
    R_{\ep} \le c_5 \|u_0\|_{L^{\infty}(\Omega)}
    + c_6 c_7 R_{\ep}^{(\theta-\frac{1}{r})(p-1)+(1-\frac{1}{q})}
\]
for any $\ep \in (0,1)$.
Thanks to \eqref{1D:crucial}, this provides a constant
$c_8 > 0$ which satisfies
$R_{\ep} \le c_8$ for all $\ep\in (0,1)$,
and hence $\|\ue(\cdot,t)\|_{L^{\infty}(\Omega)}\le c_8$
for all $t \in (0, \Tmaxe)$ and $\ep \in (0,1)$.
Inserting this into \eqref{1D:1}, we moreover find
a constant $c_9 > 0$ such that
$\|\nabla \ve(\cdot,t)\|_{L^s(\Omega)}\le c_9$
for all $t \in (0, \Tmaxe)$ and $\ep \in (0,1)$.
Therefore, this together with
the Poincar\'{e}--Wirtinger inequality thereby proves
\eqref{1D:sum}.
The latter statement also follows by means of the reasoning in
Lemma~\ref{Lem:ueLinfty}.
\end{proof}

\section{Proof of Theorem~\ref{Thm:Main}}\label{Sec:Proof}

We are in a position to
construct global weak solutions of \eqref{Sys:Main}
with the
additional regularity properties \eqref{Main_Regularity} and
the claimed boundedness feature \eqref{Main_Boundedness}.

%
%
\begin{lem}\label{Lem:GWS}
Let $n \in \N$, and let $p > 1$ and $\theta > 1$ be such that
\eqref{Main_Assumption} holds.
Then there exist a sequence $(\ep_j)_{j \in \N} \subset (0,1)$ and functions
$u,v$
such that
$\ep_j \dwto 0$ as $j \to \infty$, that $u,v$ satisfy \eqref{Main_Regularity},
that
\begin{align*}
    &\uej \wsc u \quad\mbox{in}\ L^{\infty}(0, \infty; L^{\infty}(\Omega)),
\\
    &\uej \to u \quad\mbox{a.e.\ in}\ \Omega\times (0, \infty),
\\
    &\uej \to u \quad\mbox{in}\ C^0_{\rm loc}([0, \infty); (W_0^{2,2}(\Omega))^{\ast}),
\\
    &\nabla\uej \wc \nabla u \quad\mbox{in}\ L^2_{\rm loc}([0, \infty); L^2(\Omega)),
\\
    &\vej \wsc v \quad\mbox{in}\ L^{\infty}(0,\infty; L^{\infty}(\Omega)),
\\
    &\vej \to v \quad\mbox{a.e.\ in}\ \Omega\times (0, \infty),
\\
    &\vej \to v \quad\mbox{in}\ C^0_{\rm loc}([0, \infty); (W_0^{2,2}(\Omega))^{\ast}),
\\
    &\nabla \vej \wsc \nabla v \quad\mbox{in}\ L^{\infty}(0, \infty; L^s(\Omega)),
\\
    &\nabla \vej \to \nabla v \quad\mbox{a.e.\ in}\ \Omega\times (0, \infty)
    \quad\mbox{and}
\\
    &(|\nabla\vej|^2 + \ep_j)^{\frac{p-2}{2}}\nabla\vej
    \to |\nabla v|^{p-2}\nabla v
    \quad\mbox{in}\ L^2_{\rm loc}(\Ombar \times [0, \infty))
\end{align*}
as $j\to \infty$, and that $(u,v)$ forms a global weak solution of \eqref{Sys:Main}
in the sense of Definition \ref{df1}.
Moreover, $(u,v)$ is bounded in the sense that \eqref{Main_Boundedness} holds.
\end{lem}
\begin{proof}
Noting that $\frac{s}{p-1} > n+2 > 2$ and thus the embedding
$W_{0}^{2,2}(\Omega) \embd W^{1, \frac{s}{s-(p-1)}}(\Omega)$
holds,
from \eqref{ueLinfty:sum}, \eqref{1D:sum} as well as
the first and second equations in \eqref{Sys:App}
we proceed similarly as in \cite[Lemma 2.6]{Ko1}
to find constants $c_1, c_2 > 0$ such that
\begin{align}
    &\|\ue(\cdot,t_1)-\ue(\cdot,t_2)\|_{(W_0^{2,2}(\Omega))^{\ast}}
    \le c_1|t_1 - t_2|
    \quad\mbox{and}
    \label{GWS:1}
\\
    &\|\ve(\cdot,t_1)-\ve(\cdot,t_2)\|_{(W_0^{2,2}(\Omega))^{\ast}}
    \le c_2|t_1 - t_2|
    \label{GWS:2}
\end{align}
for all $t_1, t_2 \ge 0$ and $\ep \in (0,1)$.
Hence, arguing similarly as in
\cite[Lemma 2.8]{Ko2} and
\cite[Lemma 3.5]{YL-2020},
we see from \eqref{ueLq:grad}, \eqref{ueLinfty:sum}, \eqref{1D:sum},
\eqref{GWS:1} and \eqref{GWS:2} that
there exist a sequence $(\ep_j)_{j \in \N}\subset (0,1)$ with
$\ep_j \dwto 0$ as $j\to \infty$ and functions $u,v$ which
satisfy \eqref{Main_Regularity} and the claimed convergence properties,
whence together with the assumptions \eqref{Ass1}--\eqref{Ass3}
in Section~\ref{Subsec:local} ensure that $(u,v)$ is a global weak solution of
\eqref{Sys:Main}.
The boundedness property \eqref{Main_Boundedness} can be verified by
an argument identical to that pursued in \cite[Theorem 1.2]{KoY1}.
\end{proof}

\begin{prth1.1}
All statements have been covered by Lemma~\ref{Lem:GWS}.
\qed
\end{prth1.1}

%
\smallskip
\section*{Acknowledgments}
The author thanks Professor Tomomi Yokota
for his encouragement and comments on the manuscript.


%

\small

\end{document}